\documentclass[12pt]{article}

\usepackage{paralist}
\usepackage{wrapfig}
\usepackage{graphicx}
\usepackage[none]{hyphenat}
\usepackage[english]{babel}
\usepackage[utf8]{inputenc}
\usepackage[T1]{fontenc}
\usepackage{tikz-cd}
\usepackage{epsfig}
\usepackage{graphicx}
\usepackage{epstopdf}

\usepackage{parskip}

\usepackage{authblk}
\usepackage{tikz-cd}

\usepackage{caption}

\usepackage{csquotes}
\usepackage{mathtools,amssymb,amsfonts,amsthm}

\usepackage[breaklinks,colorlinks=true]{hyperref}
\hypersetup{urlcolor=blue, citecolor=red}
\usepackage[capitalise]{cleveref} 

\theoremstyle{plain}
\newtheorem{theorem}{Theorem}
\newtheorem*{theorem*}{Theorem}
\newtheorem{lemma}[theorem]{Lemma}
\newtheorem*{lemma*}{Lemma}
\newtheorem{proposition}{Proposition}
\newtheorem*{proposition*}{Proposition}
\newtheorem{corollary}{Corollary}
\newtheorem*{corollary*}{Corollary}
\newtheorem{remark}{Remark}

\theoremstyle{definition}
\newtheorem{definition}{Definition}
\newtheorem*{definition*}{Definition}
\newtheorem{example}{Example}
\newtheorem*{example*}{Example}

\crefname{theorem}{Theorem}{Theorems}
\Crefname{theorem}{Theorem}{Theorems}
\crefname{lemma}{Lemma}{Lemmas}
\Crefname{lemma}{Lemma}{Lemmas}
\crefname{proposition}{Proposition}{Propositions}
\Crefname{Prop}{Proposition}{Propositions}
\crefname{corollary}{Corollary}{Corollaries}
\Crefname{corollary}{Corollary}{Corollaries}
\crefname{definition}{Definition}{Definitions}
\Crefname{definition}{Definition}{Definitions}
\crefname{example}{Example}{Examples}
\Crefname{example}{Example}{Examples}
\crefname{theorem}{Theorem}{Theorems}

\newtheorem*{exercise*}{Exercise}
\crefname{exercise}{exercise}{exercises}
\Crefname{exercise}{Exercise}{Exercises}

\newcommand{\dv}[2]{\frac{\dd #1}{\dd #2}}
\newcommand{\pdv}[2]{\frac{\partial #1}{\partial #2}}

\DeclarePairedDelimiter{\set}{\lbrace}{\rbrace}

\newcommand{\Cont}{\mathcal{C}}
\newcommand{\RR}{\mathbb{R}}
\newcommand{\dd}{\mathrm{d}}

\newcommand{\Reeb}{\mathcal{R}}
\newcommand*{\ann}[1]{{#1}^{\circ}} 

\newcommand*{\contr}[1]{\iota_{#1}}
\newcommand*{\VecFields}{\mathfrak{X}}

\newcommand*{\lieD}[1]{\mathcal{L}_{#1}}

\newcommand*{\orth}[1]{{#1}^{\bot}}
\newcommand*{\orthL}[1]{\prescript{\bot}{}{#1}}

\DeclareMathOperator{\supp}{supp}

\DeclareMathOperator{\pr}{pr}

\DeclarePairedDelimiter{\gen}{\langle}{\rangle}
\DeclarePairedDelimiter{\lieBr}{\lbrack}{\rbrack}
\DeclarePairedDelimiter{\jacBr}{\lbrace}{\rbrace}

\hypersetup{urlcolor=blue, citecolor=red}

\title{The nonholonomic bracket on contact mechanical systems}

\author{
\begin{center}

Manuel de Le\'on\footnote{E-mail: \href{mailto:mdeleon@icmat.es}{mdeleon@icmat.es}}
\\ Instituto de Ciencias Matem\'aticas, Campus Cantoblanco \\
 Consejo Superior de Investigaciones Cient\'ificas
 \\
C/ Nicol\'as Cabrera, 13--15, 28049, Madrid, Spain
\\
and
\\
Real Academia de Ciencias de España.
\\
C/ Valverde, 22, 28004 Madrid, Spain.

\bigskip

V{\'\i}ctor M. Jim\'enez\footnote{E-mail: \href{mailto:victor.jimenez@mat.uned.es}{victor.jimenez@mat.uned.es}}
\\  Universidad Nacional de Educación a Distancia (UNED), \\
Departamento de Matemáticas Fundamentales. \\
Calle de Juan del Rosal 10, 28040, Madrid, Spain
\end{center}
}

\begin{document}
\maketitle

\begin{abstract}
In this paper we study contact nonholonomic mechanical sys\-tems. Firstly, we present a general framework for contact constrained dynamics in such a way that the external constraints are codified into a submanifold of the contact manifold and the reaction forces are given by a distribution. We described the constrained differential equations, and we study the existence and uniqueness of solutions of these equations. We also construct two different brackets of functions which describe the evolution of the system. Furthermore, we prove that the known constrained dynamics induced by a Lagrangian are just a particular case. In addition, we develop the Hamiltonian counterpart and, in this setting, we construct another bracket, which is a natural extension of that defined by R.J. Eden, but now it is an almost Jacobi bracket since it is not satisfy the Leibniz rule.
\end{abstract}
\tableofcontents
\section{Introduction}
\sloppy
The study of nonholonomic mechanics goes back many years (see historical reviews in \cite{deLeonRACSAM} and \cite{mamaev,Borisov2002}), but they undergo a spectacular development when they could be described in a geometrical way (see \cite{Vershik1972,Koiller1992,BLMARSKRI,BLOCH1,deLeon1996d,BatesSnia}. Indeed, the use of the tools of differential geometry is a turning point, just as happened in particle mechanics with the use of symplectic geometry and Lie groups (see \cite{Arnold,Abraham1978,Libermann1987,deLeon2011}). The  book \cite{Neimark1972} is also an extraordinary source of examples and practical applications of nonholonomic systems and is always a mandatory reference on the subject.

One of the most relevant constructions in the case of nonholonomic systems is that of a bracket providing the evolution of an observable; this bracket will obviously be non-integrable, i.e., an almost-Poisson bracket. This type of bracket was first considered in two pioneering papers by Eden \cite{eden1,eden2}, and described more clearly in a paper by van der Schaft and Maschke \cite{VdSM}. Later, de Leon et al. \cite{CaLeDi99} defined in an intrinsic way a nonholonomic bracket which turned out to be the same; to do that, they used two alternative symplectic decompositions of the tangent bundle of the phase space, constructed independently by de Leon et al. \cite{CaLeDi99} and Bates and Sniatycki \cite{BatesSnia}. Recently, de León et al \cite{deleonetal} proved that it also coincided with what could be called Eden's bracket.

These considerations refer to nonholonomic systems in the symplectic environment, both in the case of Lagrangian and Hamiltonian descriptions. Of course, the time-dependent case has also been treated. But recently a great interest has arisen in the case of contact Hamiltonian systems, which reflect a dissipative nature as opposed to the conservative nature of symplectic systems \cite{cristina,Bravetti2018a,Bravetti2017, deLeon2019a, canario,deLeon2020b}. This has led to study (we believe for the first time) the case of contact nonholonomic \cite{primero} systems, in which apart from obtaining the equations of motion, a nonholonomic bracket is defined which is again almost Jacobi, i.e. it does not satisfy the Jacobi identity but neither does it satisfy the Leibniz identity.

Our first contribution of the present paper is to construct a well-defined general framework for contact constrained dynamics by considering a contact manifold $\left(P , \eta\right)$, a constrained submanifold $M \subseteq P$ and a distribution $F$ on $P$ codifying the external constrained and the reaction forces, respectively. Here, we described the associated constrained differential equations, and we extensively study properties of uniqueness and existence of solutions of these equations. In fact, we present a specific way of projecting the solutions of the unconstrained equation into solutions of the constrained equation. In this general framework, we were able to construct two different brackets, and we study the relation between them. The first one, the non-holonomic bracket, is given by an specific \textit{almost Jacobi structure} and it provides the evolution of the observables. Furthermore, it satisfies that the constant functions on the constrained submanifold $M$ are Casimirs (so, it may be restricted to $M$). For the second one, the $P-$nonholonomic bracket, we define a new projection $P$ in such a way that, under some conditions, it also projects the unconstrained dynamics into the constrained dynamics. This projection gives rise to a different almost Jacobi structure which induces the $P-$nonholonomic bracket. Finally, we prove that both brackets coincide. Furthermore, we compare this general framework with the previously one defined for contact Lagrangian dynamics \cite{primero}, and we prove that the constrained dynamics induced by a Lagrangian on a manifold $Q$, subject to a regular distribution of kinematic constraints, is just a particular case.

We also develop the Hamiltonian counterpart of non-holonomic systems from the Lagrangian side, via the Legendre transformation. However, we also prove that it may be obtained as a particular case of the general framework. Thus, we describe in detail the non-holonomic bracket for the constrained Hamiltonian system, and its associated almost Jacobi structure. Furthermore, we are able to extend the Eden construction in the case of mechanical systems \cite{eden1,eden2}, that is, when the kinetic energy comes from a Riemannian metric, and so, we obtain a bracket {\it \`a la} Eden that coincides with the previous ones. This opens new perspectives to study the possible quantization of these systems, which has been outlined in the symplectic case. Indeed, this was the motivation for Eden's pioneering work under Dirac supervision.

The paper is structured as follows. Section 2 is devoted to give an introduction to contact Hamiltonian systems, necessary for the rest of the paper. In Section 3 we develop a general framework for nonholonomic systems, in the sense that we consider a contact manifold, a constraint submanifold and a bundle of forces acting on the system. Then, we obtain conditions for the existence and unicity of dynamical solutions of the corresponding equations. Section 4 is devoted to obtain two nonholonomic brackets using two different decompositions of the tangent bundle of the phase space, and we prove that both coincide. In Section 5 we use the above procedure just to the Lagrangian case and, thus, we re-obtain the almost Jacobi bracket given in \cite{primero}. In Section 6 we apply the general scheme developed in Section 3 to the case of contact nonholonomic systems, and, in the particular case when the kinetic energy is given by a Riemannian metric, we define the corresponding Eden bracket, which coincides with the previous ones.

\section{Contact Hamiltonian systems}\label{sec:contact_jacobi}
In this section, we will give a brief introduction to contact structures and contact Hamiltonian systems. We will start with a more general geometric structure, the so-called \textit{Jacobi manifolds}~\cite{Kirillov1976,Lichnerowicz1978,Libermann1987}.

\begin{definition}\label{def:jacobi_mfd}
  A \emph{Jacobi manifold} is a triple $(M,\Lambda,E)$, where $\Lambda$ is a bivector field (namely, a skew-symmetric contravariant 2-tensor field) and $E \in \VecFields (M)$ is a vector field, so that the following identities are satisfied:
  \begin{align}
      \lieBr{\Lambda,\Lambda} &= 2 E \wedge \Lambda \\
      \lieD{E} \Lambda &= \lieBr{E,\Lambda} = 0,
  \end{align}
  where $\lieBr{\,\cdot\, , \,\cdot}$ is the Schouten–Nijenhuis bracket.
\end{definition}

By using the Jacobi bivector $\Lambda$, we may induce a vector bundle morphism between $T^{*}M$ and $TM$.
\begin{equation}\label{ham2}
    \begin{aligned}
        \sharp_{\Lambda}: T^{*}M &\to     T M\\
        \alpha &\mapsto \Lambda(\alpha, \cdot ).
    \end{aligned}
\end{equation}

On the other hand, the Jacobi structure induces (and it is characterized by) a Lie bracket on the space of functions $\Cont^\infty(M)$, which is called \emph{Jacobi bracket}.
\begin{definition}\label{def:jac_bra}
    A \emph{Jacobi bracket} $\{\cdot,\cdot\}$ on a manifold $M$ is a map $\jacBr{\cdot,\cdot}: \Cont^\infty(M) \times \Cont^{\infty}(M) \to \Cont^\infty(M)$ that satisfies
    \begin{enumerate}
        \item $(\Cont^\infty(M),\jacBr{\cdot,\cdot})$ is a Lie algebra. That is, $\jacBr{\cdot,\cdot}$ is $\RR$-bilinear, antisymmetric and satisfies the Jacobi identity:
        \begin{equation}
            \jacBr{f,\jacBr{g,h}} + \jacBr{g,\jacBr{h,f}} + \jacBr{h,\jacBr{f,g}} = 0
        \end{equation}
        for arbitrary functions $f,g,h \in \Cont^\infty(M)$.

        \item It satisfies the following locality condition: for any $f,g \in \Cont^\infty(M)$,
        \begin{equation}
            \supp(\jacBr{f,g}) \subseteq \supp(f) \cap \supp(g),
        \end{equation}
        where $\supp(f)$ is the topological support of $f$, i.e., the closure of the set in which $f$ is non-zero.
    \end{enumerate}

Namely, $(\Cont^\infty(M), \jacBr{\cdot,\cdot})$ is a local Lie algebra in the sense of Kirillov~\cite{Kirillov1976}.
\end{definition}

\noindent{Thus, given a Jacobi manifold $(M,\Lambda, E)$ we can define a Jacobi bracket as follows}
\begin{equation}\label{eq:jac_bracket_from_mfd}
    \jacBr{f,g} =\Lambda(\dd f, \dd g) + f E(g) - g E (f).
\end{equation}
Conversely, for a Jacobi bracket $\jacBr{\cdot , \cdot }$, we may consider the vector field $E$ given by,
$$E \left( f \right) = \jacBr{1,f}, \ \forall f \in \Cont^\infty \left( M \right).$$
Then, its associated Jacobi structure is defined by Eq. (\ref{eq:jac_bracket_from_mfd}). In fact, we may prove the following result.
\begin{theorem}[\cite{deLeon2019}]\label{thm:jab_bra_characterization}
    Given a manifold $M$ and a $\RR$-bilinear map $\jacBr{\cdot,\cdot}: \Cont^\infty(M) \times \Cont^{\infty}(M) \to \Cont^\infty(M)$, the following assertions are equivalent.
    \begin{enumerate}
        \item\label{item:jac_bra_local} The map $\jacBr{\cdot,\cdot}$ is a Jacobi bracket.

        \item\label{item:jac_bra_leibniz} $(M,\jacBr{\cdot,\cdot})$ is a Lie algebra which satisfies the generalized Leibniz rule
        \begin{equation}\label{eq:mod_leibniz_rule}
            \jacBr{f,gh} = g\jacBr{f,h} + h\jacBr{f,g} +  gh E(f),
        \end{equation}
        where $E$ is a vector field on $M$.

        \item\label{item:jac_bra_mfd} There is a bivector field $\Lambda$ and a vector field $E$ such that $(M,\Lambda,E)$ is a Jacobi manifold, and $\jacBr{\cdot,\cdot}$ is given as in \cref{eq:jac_bracket_from_mfd}.
    \end{enumerate}
\end{theorem}

Given a function $f \in C^\infty(M)$, its associated \textit{Hamiltonian vector field} $X_f$ is given by
\begin{equation}\label{ham}
X_f = \sharp_{\Lambda}(df) + f E,
\end{equation}
where $\sharp_\Lambda : T^*M \longrightarrow TM$ is defined by \cref{ham2}.
\vspace{0.5cm}

The characteristic distribution ${\cal C}$ of a Jacobi manifold is spanned by all the Hamiltonian vector fields, so that
$$
{\cal C} = {\rm Im} \; \sharp_\Lambda + \langle E \rangle .
$$

As we have mentioned, a particular case of Jacobi structure is provided  by a contact form.

\begin{definition}
A \emph{contact manifold} is a pair $\left(M,\eta \right)$, with $M$ a manifold of dimension $2n+1$ and $\eta$ is a $1$-form on $M$ in such a way that $\eta \wedge {(\dd\eta)}^n$ is a volume form; then,  $\eta$ will be called a \emph{contact form}.
\end{definition}
Notice that, a contact form may be equivalently defined as a nonvanishing $1-$form $\eta$ such that, for each $q \in M$, the restriction $\dd \eta_{\vert Ker \left( \eta_{q}\right)}$, is nondegenerate, i.e., it is a \textit{symplectic form} on the vector bundle $Ker \, \eta$ (see \cite{JMS}).

\begin{remark} {\rm
Notice that some authors define a contact structure as a smooth distribution $D$ on $M$ such that it is locally generated by contact forms $\eta$ via their kernels. This definition is not equivalent to ours, and is less convenient for our purposes, as explained in~\cite{deLeon2019a}. For a detailed study on contact forms, we recommend~\cite{Libermann1987,Vaisman1994}}.

\end{remark}
\begin{example}\label{48}
Consider the canonical coordinates $\left( x^{i}, y^{j} , z\right)$ on $\mathbb{R}^{2n +1}$, and define the $1-$form $\eta$ on $\mathbb{R}^{2n + 1}$ as:
\begin{equation}
    \eta = \dd z -  y^{i}\dd x^{i}.
\end{equation}
A simple computation shows that $( \mathbb{R}^{2n+1} , \eta )$ is a contact manifold. 

We also have
\begin{equation}
    \dd \eta = \dd x^{i} \wedge \dd y^{i}.
\end{equation}
Thus, the family of vector fields $\{X_{i} , Y_{i} \}_{i}$ defined by
\begin{equation}
    X_{i} = \frac{\partial}{\partial x^{i}} + y^{i} \frac{\partial}{\partial z}, \quad
    Y_{i} = \frac{\partial}{\partial y^{i}}
\end{equation}
generates the kernel of $\eta$. Furthermore,
\[\dd \eta \left( X_{i}, X_{j} \right) = \dd \eta ( Y_{i} , Y_{j} ) = 0,
 \quad \dd \eta \left( X_{i} , Y_{j} \right) = \delta^{i}_{j}.\]

\end{example}
In fact, any contact form locally looks like the contact form defined in example~\ref{48}.
\begin{theorem}[Darboux theorem]\label{Darbouxtheorem}\cite{godbillon}
Suppose $\eta$ is a contact form on a $2n+1$-dimensional manifold $M$. For each $x \in M$ there are local coordinates $(q^i, p_i, z)$ centered at $x$
such that
\begin{equation}
	 \eta = \dd z - p_i \, \dd q^i.
\end{equation}
This system of coordinates will be called \textit{Darboux coordinates}.
\end{theorem}

\noindent{For any fixed contact form $\eta$, me may define the following $2-$tensor,}
\begin{equation}
    \omega = \dd \eta + \eta \otimes \eta.
\end{equation}

We can check that $\eta$ is a contact form, if and only if, the contraction of $\omega$
\begin{equation}\label{eq:flat_iso}
    \begin{aligned}
        {\flat} : TM &\to T^* M ,\\
         v &\mapsto \omega \left( v , \cdot \right) = \contr{v}  \dd \eta + \eta (v)  \eta.
    \end{aligned}
\end{equation}
is a vector bundle isomorphism. We will also denote by $\flat:\VecFields(M)\to \Omega^1(M)$ its associated isomorphism of $\mathcal{C}^\infty(M)$-modules. The inverse of $\flat$ will be denoted by $\sharp$.\\
By using this isomorphism, one may define the \textit{Reeb vector field} $\Reeb$ as the unique vector field on $M$ satisfying that
$${\flat}(\Reeb)=\eta$$
Equivalently, $\Reeb$ is the unique vector field on $M$ such that
\begin{equation}
	\contr{\Reeb}  \dd \eta = 0, \quad  \contr{\Reeb}\, \eta = 1.
\end{equation}

Observe that in Darboux coordinates, the Reeb vector field $\Reeb$ is the partial derivative $\pdv{}{z}$. In fact,
\begin{itemize}
    \item $\flat \left( \frac{\partial}{\partial p_{i}}\right) = - \dd q^{i} $
    \item $\flat \left( \frac{\partial}{\partial q^{i}}\right) =  \dd p_{i} - p_{i}\eta $
\end{itemize}
Then,
\begin{itemize}
    \item $\sharp \left( \dd q^{i}  \right) = - \frac{\partial}{\partial p_{i}} $
    \item $\sharp \left( \dd p_{i}  \right) =  \frac{\partial}{\partial q_{i}} + p_{i} \frac{\partial}{\partial z} $
    \item $\sharp \left( \dd z  \right) = \frac{\partial}{\partial z} -  p_{i} \frac{\partial}{\partial p_{i}} $
\end{itemize}
The notion of \textit{contact Hamiltonian vector field} $X_H$ for a  \textit{Hamiltonian function} $H: M \rightarrow \mathbb{R}$, is defined by the formula
\begin{equation}\label{eq:hamiltonian_vf_contact}
    {\flat} (X_H) = \dd H - (\Reeb (H) + H) \, \eta,
\end{equation}
Then, it is locally expressed as follows
\begin{equation}\label{Hamiltneq1}
    X_{H} =\pdv{H}{p_i}\pdv{}{q^i}  -\left(\pdv{H}{q^i}  +p_{i} \pdv{H}{z}\right) \pdv{}{p_{i}}    + \left( p_{i} \pdv{H}{p_{i}}-H\right)\pdv{}{z}
\end{equation}
Therefore, its integral curves satisfy the \textit{contact Hamilton equations}
\begin{equation}\label{Hamiltoneqs340}
\begin{dcases}
 \frac{\dd q^{i}}{\dd t} &=    \pdv{H}{p^{i}}\\
 \frac{\dd p_{i}}{\dd t} &=  - \pdv{H}{q^{i}}  - p_{i} \pdv{H}{z}\\
 \frac{\dd z}{\dd t} &=   p_{i} \pdv{H}{p_{i}}-H
 \end{dcases}
\end{equation}

We may define the Jacobi structure $(M,\Lambda,E)$ associated to the contact manifold $(M,\eta)$ as follows
\begin{subequations}
  \begin{equation}\label{eq:contact_jacobi}
    \Lambda(\alpha,\beta) =
    -\dd \eta (\sharp (\alpha), \sharp (\beta)), \quad
    E = - \Reeb.
  \end{equation}
\end{subequations}

One can prove that the Hamiltonian vector fields $X_f$ defined with both structures, \cref{ham2} for Jacobi or and \cref{eq:hamiltonian_vf_contact} for contact, coincide. Equivalently, the notion of contact Hamiltonian vector field $X_H$ for a Hamiltonian function $H$, can be alternatively obtained by the formula
\begin{equation}\label{eq:hamiltonian_vf_jacobi}
X_H = \sharp_{\Lambda} (dH) - H \mathcal{R}.
\end{equation}

Observe that, in the case of contact manifolds, the map $\sharp_{\Lambda}$ can be written directly in terms of the contact structure~\cite[Section~3]{deLeon2019a} as:
    \begin{equation}
        \sharp_{\Lambda} (\alpha) =
        \sharp (\alpha) - \alpha(\Reeb) \Reeb.
    \end{equation}

\begin{remark}{\rm It is important do not confuse $\sharp$ (the inverse of $\flat$), with $ \sharp_{\Lambda}$ (the morphism from $T^{*}M $ to $TM$ associated to $\Lambda$).
}
\end{remark}

Using the contact form $\eta$, the tangent bundle of $M$ may be decomposed into a Whitney sum as
\begin{equation}\label{eq:contact_sum_decomposition}
    TM = \ker \eta \oplus \ker \dd\eta,
\end{equation}
Furthermore, it is easy to realize that $ \ker \dd \eta$ is generated by the Reeb vector field $\Reeb$. Notice that, $\sharp_{\Lambda}$ is not an isomorphism, since
$$ \ker \sharp_{\Lambda}  = \gen{\eta}$$
Therefore,
$$ \sharp_{\Lambda} \left( T^{*}M\right) = \ker  \eta$$
Thus, the above decomposition (\ref{eq:contact_sum_decomposition}) may be written in terms of the Jacobi structure $\Lambda$.

\begin{definition}\label{def:subspace_position}
Let $\mathcal{D}_x \subseteq T_x M$ be a subspace, where $x\in M$. We say that $\mathcal{D}_x$ is
    \begin{enumerate}
        \item \emph{Horizontal} if $\mathcal{D}_x \subseteq \ker  \eta_x$.
        \item \emph{Vertical} if $\Reeb_x \in \mathcal{D}_x $.
        \item \emph{Oblique} otherwise. By a dimensional counting argument, this is equivalent to $\mathcal{D}_x =(\mathcal{D}_x \cap \ker  \eta_x)\oplus \gen{\Reeb_x + v_x}$, with $v_x \in \ker  \eta_x \setminus \mathcal{D}_x$.
    \end{enumerate}
We say that a distribution $\mathcal{D}$ is \emph{horizontal}/\emph{vertical}/\emph{oblique} if $\mathcal{D}_x$ is \emph{horizontal}/\emph{vertical}/\emph{oblique}, for every point $x$, respectively.
\end{definition}

Let us consider a distribution $\mathcal{D} \subseteq TM$. Then, we may define the following notions of complement with respect to $\omega$:
\begin{equation}
    \begin{aligned}
        \orth{\mathcal{D}}  &=
          \set{v \in TM \mid \omega(w,v) = {\flat}(w)(v) = 0,\, \forall w\in \mathcal{D}}
          = \ann{({\flat}(\mathcal{D}))},\\
        \orthL{\mathcal{D}} &=
          \set{v \in TM\mid \omega(v,w) = 0, \forall w \in \mathcal{D}} = \flat^{-1}(\ann{\mathcal{D}}).
    \end{aligned}
\end{equation}
Observe that, considering $\omega_{-}$ as the $2-$tensor associated to the contact form $-\eta$, then we have, $\omega \left( w,v \right) = \omega_{-} \left( v , w \right)$, for all $v,w \in TM$ and, therefore,
$$\orth{\mathcal{D}_{-}}=\orthL{\mathcal{D}},  $$
where 
$$\orth{\mathcal{D}_{-}} =  \set{v \in TM \mid \omega_{-}(w,v)  = 0,\, \forall w\in \mathcal{D}}$$
We have
\begin{equation}\label{complements_cancel}
    \orthL{(\orth{\mathcal{D}})} =
    \orth{(\orthL{\mathcal{D}})} = \mathcal{D}.
\end{equation}
Furthermore, we may interchange sums and intersections, i.e,
\begin{subequations}\label{eq:complement_intersections}
    \begin{align}
        \orth{(\mathcal{D}\cap \Gamma)} &= \orth{\mathcal{D}} + \orth{\Gamma},\\
        \orth{(\mathcal{D} +  \Gamma)} &= \orth{\mathcal{D}}  \cap \orth{\Gamma},
    \end{align}
\end{subequations}
for all distributions $\mathcal{D}$ and $\Gamma$ on $M$ (for the left complement, we obtain similar identities that for the right case).
\begin{lemma}[\cite{primero}]\label{lem:rl_complement}
    Let $\mathcal{D}$ be a distribution on $M$.
    \begin{itemize}
        \item  If $\mathcal{D}$ is horizontal, then
        \begin{equation}\label{eq:horizontal_orth}
            \orth{\mathcal{D}} = \orthL{\mathcal{D}} = \set{v \in TM \mid \dd \eta(v,w) = 0, \, \forall w \in \mathcal{D}}.
        \end{equation}
        Hence, $\orth{\mathcal{D}}$ is vertical.

        \item  If $\mathcal{D}$ is vertical, then
        \begin{equation}
            \orth{\mathcal{D}} = \orthL{\mathcal{D}} = \set{v \in \ker \eta \mid \dd \eta(v,w) = 0, \, \forall w \in (\mathcal{D} \cap \ker \eta)}.
        \end{equation}
        Hence, $\orth{\mathcal{D}}$ is horizontal.
    \end{itemize}
\end{lemma}

\vspace{1cm}

In what follows, we will omit the word \emph{contact} and call $X_H$ simply a \emph{Hamiltonian vector field}. The triple $(M,\eta,H)$ will be called \emph{contact Hamiltonian system}.
\begin{proposition}\label{prop213123}
Let $H$ be a Hamiltonian function. The following statements are equivalent:
\begin{enumerate}
\item $X_{H}$ is the Hamiltonian vector field of $H$.

\item It satisfies that
\begin{subequations}\label{eqs:ham_vfs_characterization2}
    \begin{align}
        \eta(X_H) &= -H, \\
        \iota_{X_{H}} \dd \eta_{\vert \ker \eta} &= \dd H_{\vert \ker \eta}\label{eq:ham_vf_conf_contactomorphism2}
    \end{align}
\end{subequations}
\item It satisfies that
\begin{subequations}\label{eqs:ham_vfs_characterization}
    \begin{align}
        \eta(X_H) &= -H, \label{eqs:ham_vfs_characterizationa} \\
        \lieD{X_H} \eta &= a \eta \label{eq:ham_vf_conf_contactomorphism},
    \end{align}
\end{subequations}
for some function $a$. Notice that this implies that
$a = -\Reeb(H)$.
\end{enumerate}
\end{proposition}

By Cartan's formula, formulas (\ref{eqs:ham_vfs_characterizationa}) and (\ref{eq:ham_vf_conf_contactomorphism}) imply the following identity
\begin{subequations}
    \begin{align}
        \contr{X_H} \dd \eta &= \dd H - \Reeb(H) \eta.
    \end{align}
\end{subequations}
Thus, in contrast with the symplectic case, the energy and the phase space volume are not preserved. In fact, denoting the volume form by $\Omega = \eta \wedge {(\dd\eta)}^n$, the \textit{energy and volume dissipation} are given by the following formulas:
\begin{itemize}
    \item[i)] 
    \begin{equation}
        \lieD{X_H} H = - \Reeb(H) H.
    \end{equation}

    \item[ii)]

     \begin{equation}\label{estadeaqui234}
        \lieD{X_H} \eta = - \Reeb(H) \eta.
    \end{equation}

\item[iii)]
\begin{equation}
        \lieD{X_H} \Omega=
         - (n+1) \Reeb (H) \Omega.
    \end{equation}
\end{itemize}

\section{General framework for contact Hamiltonian constrained dymanics}\label{genframework24}

Assume a $2n+1-$dimensional contact manifold $\left( P, \eta\right)$ equipped with a Hamiltonian function $H : P \rightarrow \mathbb{R}$. The external constraints will be codified into an $r-$dimensional embbeded submanifold $M$ of $P$, called the \textit{constraint submanifold}. Furthermore, we will consider a regular distribution $F$ on $P$ along $M$, namely, a vector subbundle of $T_{ M} P$ (the space of tangent vectors to $P$ at points in $M$). In general, the distribution $F$ characterizes the \textit{reaction forces}, via the annihilator $\ann{F}$.

Thus, we are interested in finding vector fields $X\in \mathfrak{X}\left( P \right)$ such that $X$ is tangent to the constraint submanifold $M$ and satisfy the following equation
$$ \flat \left(X\right) = \dd H - \left(H + \Reeb\left(H\right)\right)\eta + \lambda_{a}\Psi^{a},$$
where $\Psi^{a}$ are reaction forces and $\lambda_a$ are some Lagrange multipliers to be determined. In other words, we are looking for solutions of the following equation

\begin{equation}
    \begin{dcases}\label{6.4}
        \flat \left(X\right) - \dd H + \left(H + \Reeb\left(H\right)\right)\eta \in \ann{F} \\
        X_{\vert M} \in  \mathfrak{X}\left(M \right).
    \end{dcases}
\end{equation}

From now on, we will consider $\left( q^{i} , p_{i} , z\right)$ as a (local) system of Darboux coordinates for $\eta$. Let $\set{\Psi^{a} \ =\ \Phi^{a}_{i}\dd q^{i} \ + \ \Psi^{a}_{i}\dd p_{i} \ + \ \mu^{a} \dd z} $ be a (local) basis of $1-$forms of $\ann{F}$, and a vector field $Y$ on $P$ such that
\begin{equation}\label{12}
 \flat \left( Y \right) \in \ann{F}.
 \end{equation}
Then, there exist some (local) functions $\lambda_{a}$ such that
$$ \flat\left( Y \right)  \ = \  \lambda_a\Psi^{a} \ =\ \lambda_a\Phi^{a}_{i}\dd q^{i} \ + \ \lambda_a\Psi^{a}_{i}\dd p_{i} \ + \ \lambda_a\mu^{a} \dd z.$$
Observe that $\eta = \dd z -  p_{i}{\dd q}^{i}$ and, hence, $\dd \eta ={\dd q}^{i} \wedge {\dd p}_{i}$. Thus,
\begin{eqnarray*}
\lambda_a\Psi^{a}_{i} & = & \{ \flat \left( Y \right) \} \left( \pdv{}{p_i}\right) = \iota_{Y} \dd \eta \left( \pdv{}{p_i}\right)  + \eta \left( Y \right)  \eta \left( \pdv{}{p_i}\right)\\
&=& -\iota_{\pdv{}{p_i}}\dd \eta  \left( Y \right) = Y \left(q^{i} \right)
\end{eqnarray*}
Analogously, contracting by $\pdv{}{z}$, we obtain that $\eta\left( Y \right) = \lambda_a\mu^{a}$ and 
$$Y \left( z \right) = \lambda_{a}\left[\mu^{a}+p_{i}\Psi^{a}_{i}\right]$$
Finally,
\begin{eqnarray*}
\lambda_a\Phi^a_i & = & \{ \flat \left( Y \right) \} \left( \pdv{}{q^i}\right) = \iota_{Y} \dd \eta_{Q} \left( \pdv{}{q^i}\right)  + \eta_{Q} \left( Y \right)  \eta_{Q} \left( \pdv{}{q^i}\right)\\
&=&  -   Y \left( p_{i}\right) - \lambda_{a}\mu^{a}p_{i}
\end{eqnarray*}
Then,
$$Y\left( p_{i}\right) = -\lambda_{a}\left[\mu^{a}p_{i}+\Phi^{a}_{i}\right].$$
Next, a vector field $X$ solving \cref{6.4} should satisfy that $\flat\left( X- X_{H}\right) \in \ann{F}$, where $X_{H}$ is the solution of the contact Hamiltonian system (\ref{Hamiltoneqs340}). Therefore,
\begin{eqnarray*}\label{nonholHamiltneqgeneralized}
    X &=&\left[\pdv{H}{p_i}+ \lambda_a\Psi^{a}_{i}\right]\pdv{}{q^i} - \left[\pdv{H}{q^Pi}  +p_{i} \pdv{H}{z} + \lambda_{a}\left(\mu^{a}p_{i}+\Phi^{a}_{i}\right) \right]\pdv{}{p_{i}}\\
    &+&\left[\left( p_{i} \pdv{H}{p_{i}}-H+ \lambda_{a}\left(\mu^{a}+p_{i}\Psi^{a}_{i}\right)\right)\right]\pdv{}{z}
\end{eqnarray*}
The Lagrange multipliers $\lambda_{a}$ will be determined from the tangency condition to $M$, and, in general, they will not be unique. 
Notice that, $M$ may be locally described in terms of independent (local) functions $\set{\varphi^{a}}$ on $P$
in the following way
$$M = \set{x \in P \ \mid \ \varphi^{a}\left( x \right) = 0, \ \forall a } = \varphi^{-1}\left(0\right),$$
where $\varphi = \left( \varphi^{1}, \dots , \varphi^{2n-1-r}\right)$. These functions $\varphi^{a}$ will be called \textit{constraint functions} defining $M$.
\begin{theorem}\label{18.4}
Let be a Hamiltonian function $H: P \rightarrow \mathbb{R}$, and $M \subseteq P$ a constraint submanifold. Let $X$ be a vector field on $P$ satisfying \cref{6.4}. Then, the integral curves $\xi$ of $X$ are solutions of the following equations,
\begin{align}
    \begin{cases}\label{generalizedeq:Hamiltonnonholonomic_herglotz_eqs_coordsMechatype}
         \frac{\dd q^{i}}{\dd t} &=   \pdv{H}{p^{i}} + \lambda_a\Psi^{a}_{i}\\
 \frac{\dd p_{i}}{\dd t} &=  - \pdv{H}{q^{i}}  -p_{i} \pdv{H}{z} - \lambda_{a}\left(\mu^{a}p_{i}+\Phi^{a}_{i}\right)\\
 \frac{\dd z}{\dd t} &=   p_{i} \pdv{H}{p_{i}}-H + \lambda_{a}\left(\mu^{a}+p_{i}\Psi^{a}_{i}\right)\\
        \varphi^{a}\left(\xi\right) & = 0, \ \forall a
    \end{cases}
\end{align}
\end{theorem}
\cref{generalizedeq:Hamiltonnonholonomic_herglotz_eqs_coordsMechatype} will be called \textit{constrained Herglotz equations}. Furthermore, in case of existence and uniqueness, the solution $X$ to \cref{6.4} will be called \textit{constrained Hamiltonian vector field} and denoted by $X_{H,M}$.

\begin{definition}\label{definitionconstraind23}
\rm
A \textit{constrained Hamiltonian system} is given by a quintuple $\left(P, \eta, H, M, F \right)$, where $H : P \rightarrow \RR$ is a Hamiltonian function on a contact manifold $P$ with contact form $\eta$, $M$ is an embbeded submanifold of $P$, and $F$ is a regular distribution on $P$ along $M$.
\end{definition}

In general, the existence and uniqueness are not guaranteed.

Let be the vector subbundle of $T_{M}P$ defined by $\sharp \left( \ann{F}\right) = \orthL{F}$, with $\sharp = \flat^{-1}$. Then, for any two solution $X$ and $Y$ to \cref{6.4}, the difference $X-Y$ is tangent to $\orthL{F} \cap T  M$. In other words, we have proved the following result:

\begin{proposition}\label{GeneralizedYoootraproposiciondelcarajomas}
The uniqueness of solutions of \cref{6.4} is equivalent to the condition
\begin{equation}\label{uniqueness4562}
\orthL{F} \cap T M= \{0\}.    
\end{equation}

\end{proposition}

Let us assume that $rank\left(F \right) = dim \left( M \right)$. Then,
\begin{equation}\label{41.2}
\orthL{F} \oplus T M = T_{M}P,
\end{equation}
where $T_{M}P$ consists of the tangent vectors to $P $ at points of $M $. Therefore, we may consider these two projectors:
\begin{subequations}\label{projectors2.2}
\begin{align}
\mathcal{P}: T_{M}P &\rightarrow T M,\\
\mathcal{Q} : T_{M}P &\rightarrow \orthL{F}.
\end{align}
\end{subequations}
Let be $X = \mathcal{P} \left( {X_{H}}_{\vert M} \right)$, where $X_{H}$ is the Hamiltonian vector field associated to $H$ (see \cref{eq:hamiltonian_vf_contact}). Then, by definition $X \in  \mathfrak{X} \left(M\right)$. On the other hand, at points of $M$, we have
\begin{align*}
 &\flat\left(X\right) - \dd H + \left(H + \Reeb\left(H\right)\right)\eta \\ &=
\flat \left(X_{H} - \mathcal{Q}\left( X_{H} \right) \right) - \dd H + \left(H + \Reeb\left(H\right)\right)\eta\\ & =
 -\flat \left(\mathcal{Q}\left( X_{H} \right) \right) \in \ann{F}
\end{align*}
Therefore, by uniqueness, $\mathcal{P} \left( {X_{H}}_{\vert M} \right) = X_{H , M}$ is the solution of \cref{6.4}.

\begin{corollary}\label{equaldimensions34}
If $rank\left(F \right) = dim \left( M \right)$, then condition (\ref{uniqueness4562}) implies both the existence and uniqueness of solutions.
\end{corollary}

Nevertheless, this corollary is not a necessary condition for the existence of solutions.

\begin{lemma}\label{gen37.e}
\cref{6.4} admits solutions if, and only if,
$$ \dd H - \left( H + \Reeb\left( H \right)\right)\eta \in \ann{\left(\orth{TM} \cap F \right)}$$

\begin{proof}
Notice that $\flat \left( X_{H, M}\right) \in \flat \left( TM \right) = \ann{\left(\orth{TM}\right)}$. Then, there exists $\Phi \in \ann{F}$ such that
\begin{small}
$$\dd H_{\vert M} - \left( H + \Reeb\left(H \right) \right)\eta_{\vert M}= \flat \left( X_{H, M}\right) + \Phi \in \ann{\left(\orth{TM}\right)} \ + \ \ann{F} = \ann{\left( \orth{TM} \cap F \right)}$$
\end{small}
\end{proof}
\end{lemma}

Furthermore, in the special case in which the Reeb vector field $\Reeb$ is tangent to the constraint submanifold ($TM$ is vertical) or $\eta$ is in $\ann{F}$ ($F$ horizontal), we may improve the above result.

\begin{corollary}\label{37.e}
Assume that $TM$ is vertical ($\Reeb_{\vert M} \in \mathfrak{X}\left( M \right)$) or $F$ is horizontal ($ \eta \in \ann{F}$). Then, \cref{6.4} admits solutions if, and only if,
$$ \dd H \in \ann{\left(\orth{TM} \cap F \right)}$$
\end{corollary}

\begin{remark}{\rm
Notice that
\begin{equation}
X_{a} = \sharp \left( {\Psi}^{a}\right) = \Psi^{a}_{i}\frac{\partial}{\partial q^{i}} - \left[\Phi^{a}_{i} + \mu^{a}p_{i}\right] \frac{\partial}{\partial p_{i}} + \left[\mu^{a}+\Psi^{a}_{k}p_{k}\right]\frac{\partial}{\partial z}
\end{equation}
is a (local) basis of sections of the subbundle $\orthL{F}$. Then, a vector field $Y =Y^{b}X_{b} $ is tangent to $M$ if, and only if,
\[ Y^{b} \left[\Psi^{b}_{i}\frac{\partial \varphi^{a}}{\partial q^{i}} - \left[\Phi^{b}_{i} + \mu^{b}p_{i}\right] \frac{\partial \varphi^{a}}{\partial p_{i}} + \left[\mu^{b}+\Psi^{b}_{k}p_{k}\right]\frac{\partial \varphi^{a}}{\partial z}\right] = 0, \ \ \forall a\]
Therefore, the matrix $\mathcal{C}$ with coefficients,
\begin{equation}\label{matrixcoeff34}
\mathcal{C}_{ab} = \Psi^{b}_{i}\frac{\partial \varphi^{a}}{\partial q^{i}} - \left[\Phi^{b}_{i} + \mu^{b}p_{i}\right] \frac{\partial \varphi^{a}}{\partial p_{i}} + \left[\mu^{b}+\Psi^{b}_{k}p_{k}\right]\frac{\partial \varphi^{a}}{\partial z}
\end{equation}
codifies the property (\ref{uniqueness4562}) of uniqueness. Indeed, \textit{the condition of uniqueness (\ref{uniqueness4562}) is satisfied if, and only if, the rank of $\mathcal{C}$ is maximal and equal to $2n+1-rank{\left(F\right)}$.} Thus, as a consequence, a necessary condition for uniqueness of solutions of \cref{6.4} is that
$$rank{\left(F \right)} \geq dim{\left( M \right)}.$$
Assume now that $rank{\left(F \right)} = dim{\left( M \right)}$, i.e., \cref{6.4} has the property of uniqueness and existence of solutions (Corollary \ref{equaldimensions34}).
}
\end{remark}

Let $\{ \Theta_{i}\}$ be a (local) basis of $T M$. Then, the restriction of any vector field $Y \in \mathfrak{X}\left(P\right)$ to $M$ may be locally written as,
\[ Y_{\vert M} = Y^{i}\Theta_{i} + \lambda^{a}X_{a}.\]
Then, applying $T \varphi^{b}$ we have that
\[ T \varphi^{b} \left( Y \right) = Y \left( \varphi^{b}\right) = \lambda^{a} \mathcal{C}_{ba},\]
Thus, we obtain
\begin{equation}\label{19}
\lambda^{a} = Y \left( \varphi^{b}\right) \mathcal{C}^{ba},
\end{equation}
where $\mathcal{C}^{ba}$ are the coefficients of the inverse matrix of $\mathcal{C}$. Hence,
\begin{itemize}
\item $\mathcal{Q} \left( Y_{\vert M} \right) = Y \left( \varphi^{b}\right)\mathcal{C}^{ba}X_{a}.$

\item $\mathcal{P} \left( Y_{\vert M} \right) = Y_{\vert M} -  Y \left( \varphi^{b}\right)\mathcal{C}^{ba}X_{a}.$
\end{itemize}
Therefore, an \textit{explicit (local) expression} of the solution $X_{H, M}$ is given by
\begin{equation}\label{20.2}
X_{H, M} = \left(X_{H}\right)_{\vert M} - X_{H} \left( \varphi^{b}\right)\mathcal{C}^{ba}X_{a}
\end{equation}

Notice that, for the case of $rank \left( F \right) > dim {\left(M\right)}$, we may still obtain the unique solution for \cref{6.4} by projecting the Hamitonian vector field $X_{H}$. In this case $\orthL{F}\oplus TM \lneq T_{M}P$. However, by existence condition given in Lemma \ref{gen37.e}, we have
$$X_{H} = \sharp \left(\dd H_{\vert M}\right) - \left( H + \Reeb \left( H \right) \right) \Reeb  \in \orthL{F}\oplus TM.$$
Then, we may consider the projection $X_{H,M} =\mathcal{P}\left( {X_{H}}_{\vert M}\right) $ and it will be a solution for \cref{6.4}. Thus, we may calculate the solution $X_{H.M}$ to \cref{6.4}, by projecting the Hamiltonian vector field $X_{H}$ via $\mathcal{P}$, even in \textit{the most general case in which there is uniqueness and existence of solutions to \cref{6.4}}.

From now on, we will assume existence and uniqueness of solutions. Then, by applying $\flat$, we have two projections, $\overline{\mathcal{P}}$ and $\overline{\mathcal{Q}}$, of the vector subbundle, $ \overline{TM} \oplus \ann{F}$ of $T^{*}_{M}P$ such that $\overline{TM} = \flat \left( T M\right)$. Then, for all $a_{x} \in T^{*}_{x} P$, with $x \in M$, the associated projections $\overline{\mathcal{Q}}: \overline{TM} \oplus \ann{F} \rightarrow \ann{F}$ and $\overline{\mathcal{P}}: \overline{TM} \oplus \ann{F} \rightarrow \overline{TM}$ are given by
\begin{itemize}
\item $\overline{\mathcal{Q}} \left( a_{x}\right) = \flat \left(\mathcal{Q}\left(\sharp \left( a_{x}\right)\right)\right)$.

\item $\overline{\mathcal{P}} \left( a_{x}\right) = \flat\left(\mathcal{P}\left(\sharp \left( a_{x}\right)\right)\right)$.
\end{itemize}
\begin{theorem}\label{21.2}
Let $X$ be a vector field on $P$. Then, $X$ satisfies \cref{6.4}, if, and only if,
\begin{equation}
    \begin{dcases}\label{22}
        \flat\left( X \right) = \overline{\mathcal{P}} \left(  \dd H + \left(H + \Reeb \left(H\right)\right)\eta\right) \\
        X_{\vert M} \in  \mathfrak{X} \left(M \right).
    \end{dcases}
\end{equation}
\end{theorem}

Notice that, by Lemma \ref{gen37.e}, we have that
$$\dd H - \left(H + \Reeb \left(H\right)\right)\eta \in \ann{\left(\orth{TM} \cap F \right)} = \overline{TM} \oplus \ann{F} $$
Then, it may be projected by $\overline{\mathcal{P}}$.

At this point, we would like to highlight a relevant family of constrained Hamiltonian systems.

\begin{definition}\label{mechanicalcondit2343}
\rm
A \textit{constrained Hamiltonian system} $\left( P, \eta, H, M, F \right)$ satisfies the \textit{mechanical condition} if  $\dd H_{\vert M} \in \ann{\orthL{F}}$.
\end{definition}
This condition could looks ``\textit{artificial}''.
However, taking into account the splitting $T_{M}P = \orthL{F} \oplus T M $, the mechanical condition can be interpreted as the fact that $\dd H_{\vert M }$ depends only on $TM$, as a fiberwise linear map on $T_{M}P$. In other words, the derivative of the Hamiltonian at points in the constraint submanifold $M$ depends only on the constraint manifold. We will see that it is a natural condition when we study Hamiltonian and Lagrangian constrained dynamics in the next sections. In fact, we will prove that all the Lagrangian and Hamiltonian systems of \textit{mechanical type} satisfy these mechanical condition.

\begin{theorem}\label{25.2}
Let $\left( P , \eta, H, M , F \right)$ be a constrained Hamiltonian system, with $\Reeb$ tangent to $F$. A vector field $X$ on $P$ satisfies
$$\flat\left(X\right) - \dd H + \left(H + \Reeb\left(H\right)\right)\eta \in \ann{F}$$
if, and only if,
\begin{equation}
    \begin{dcases}\label{26.2}
    \mathcal{L}_{X}\eta + \Reeb\left(H\right) \eta \in \ann{F} \\
        \eta \left(  X\right) = - H.
    \end{dcases},
\end{equation}
\begin{proof}
Assume that $X $ fulfills the condition
\begin{equation}\label{aux:equation234}
\flat\left(X\right) - \dd H + \left(H + \Reeb\left(H\right)\right)\eta \in \ann{F}
\end{equation}
Then, applying $\Reeb$ ($\Reeb \in F$), we have that
$$\eta\left( X \right) = - H$$
Then, the above equation reduces to
\begin{equation*}
\contr{X}\dd \eta - \dd \eta\left( X \right) + \Reeb\left(H \right)\eta \in \ann{F}
\end{equation*}
In other words,
$$\mathcal{L}_{X}\eta+ \Reeb \left(H\right) \eta \in \ann{F}$$
The converse is proved by reversing the above arguments.
\end{proof}
\end{theorem}

Thus, this result permits us to give another equivalent geometric equation to \cref{6.4}. Notice the similarity with the result on energy dissipation of the Hamiltonian vector field (\ref{estadeaqui234}). In fact, we know that contact Hamiltonian vector fields model the dynamics of dissipative systems and, contrarily to the case of symplectic Hamiltonian systems, the evolution does not preserve the energy, the contact form and the volume, i.e.,
\begin{align*}
    \lieD{X_H} H  &= - \Reeb(H) H, \\
    \lieD{X_H} \eta &= - \Reeb(H) \eta.\\
    \lieD{X_H} \Omega &=
         - (n+1) \Reeb (H) \Omega.
\end{align*}

This result may be naturally generalized to the case of non-holonomic constraints by using the projectors $\mathcal{P}$ and $\mathcal{Q}$.
\begin{proposition}\label{39}
Let $\left( P , \eta, H, M , F \right)$ be a constrained Hamiltonian system satisfying the mechanical condition. Assume that $\orthL{F} \cap T M= \{0\}$ (uniqueness) and $rank{\left( F \right)} = dim{\left( M \right)}$ (existence). Then, a vector field $X$ solving the constrained Herglotz equations \cref{6.4} satisfies that

\begin{subequations}
    \begin{align}
    \mathcal{L}_{X} H&=
          - \mathcal{R} \left( H\right)H- \lieD{\mathcal{Q}(X_{H})} H,\\
        \mathcal{L}_{X} \eta &=
          - \mathcal{R} \left( H\right)\eta- \lieD{\mathcal{Q}(X_{H})} \eta,\\
        \mathcal{L}_{X} \tilde{\eta} &=
          - \frac{\lieD{\mathcal{Q}(X_{H})} \eta}{H} \\
        \lieD{X} \Omega &= -(n+1)\Reeb (H) \Omega  - \eta \wedge {\dd \eta}^{(n-1)} \wedge \dd \mathcal{L}_{\mathcal{Q}(X_{H})} \eta\\
        \lieD{X} \tilde{\Omega} &= \tilde{\eta} \wedge {\dd \tilde{\eta}}^{(n-1)} \wedge \dd \mathcal{L}_{\mathcal{Q}(X_{H})} \tilde{\eta}
    \end{align}
\end{subequations}
where $\tilde{\eta} = \eta/H$, assuming that $H$ does not vanish, $\Omega = \eta \wedge {(\dd \eta)}^{n}$ is the contact volume element and $\tilde{\Omega} = \tilde{\eta} \wedge {(\dd \tilde{\eta})}^{n}$.

Furthermore, if $\Reeb \in F$, we have that $\lieD{\mathcal{Q}\left(X_{H}\right) }\eta \in  \ann{F} $.
\begin{proof}
The first item follows from $\mathcal{P}\left( X_{H}\right) = X $ and $\mathcal{L}_{X_{H}} H = -\mathcal{R} \left( H\right)H$. The second item can be proved in an analogous way. Then, by using Theorem \ref{25.2}, we have that $\lieD{\mathcal{Q}\left(X_{H}\right) }\eta \in  \ann{F} $.

For the third assertion, we use the product rule
\begin{eqnarray*}
\mathcal{L}_{X} \tilde{\eta} &=& \frac{1}{H}  \mathcal{L}_{X} \eta + \eta X \left( \frac{1}{H}\right)\\
&=& -\frac{1}{H}  \left(\mathcal{R} \left( H\right)\eta + \lieD{\mathcal{Q}(X_{H})} \eta \right) + \eta X \left( \frac{1}{H}\right)\\
&=& -\frac{1}{H}  \left(\mathcal{R} \left( H\right)\eta + \lieD{\mathcal{Q}(X_{H})} \eta \right) + \eta \frac{  \mathcal{R} \left( H\right)H + \lieD{\mathcal{Q}(X_{H})} H,}{H^{2}}\\
&=& -\frac{\lieD{\mathcal{Q}(X_{H})} \eta}{H} + \eta \frac{\lieD{\mathcal{Q}(X_{H})} H,}{H^{2}}\\
&=& -\frac{\lieD{\mathcal{Q}(X_{H})} \eta}{H}
\end{eqnarray*}
For the last equality, we have used mechanical condition. For the fourth claim, we perform the following computation.
\begin{small}
\begin{equation*}
    \begin{aligned}
        \lieD{X}( \eta \wedge (\dd \eta)^n )  &=
        - \Reeb(H) \Omega  +
        n \eta \wedge (\dd \eta)^{n-1} \wedge \dd \lieD{X}\eta  \\
        & =
        -(n+1) \Reeb(H) \eta \wedge (\dd \eta)^n -
        \eta \wedge {\dd \eta}^{(n-1)} \wedge \dd \mathcal{L}_{\mathcal{Q}(X_{H})} \eta,
    \end{aligned}
\end{equation*}
\end{small}
An analogous computation proves the fifth claim.
\end{proof}

\end{proposition}

\section{Non-holonomic brackets}

This section is devoted to present the non-holonomic bracket of functions, which may be defined on a constrained Hamiltonian system $\left(P, \eta, H,M,F\right)$. We will present different ways to construct a non-holonomic bracket and we study the relation between them.

First, we will construct a \textbf{non-holonomic bracket} for constrained Hamiltonian systems $\left(P,\eta, H, M, F\right)$ satisfying the mechanical condition, which is a natural generalization of the known one in the special of constrained Lagrangian contact dynamics \cite{primero}. To do this we will also assume conditions of uniqueness and existence of Proposition \ref{GeneralizedYoootraproposiciondelcarajomas} and Corollary \ref{equaldimensions34}. Then, me may consider an alternative splitting of the cotangent bundle 
\begin{equation}\label{43.2}
T^{*}_{M} P = \ann{\left(\orthL{F}\right)}\oplus \ann{T M }
\end{equation}
where the projections are the adjoint operators $\mathcal{P}^*$ and $\mathcal{Q}^*$ of the projections $\mathcal{P}^* :T^{*}_{M} P \rightarrow \ann{\left(\orthL{F}\right)}$ and $\mathcal{Q}^* :T^{*}_{M} P\rightarrow \ann{T M }$, respectively.

Thus, we can define, along $M$, the following geometric objects:
\begin{align}\label{almostjac24.2}
    \Reeb_{H,M} &=  \mathcal{P} \left({\Reeb}_{\vert M}\right),\\
    \Lambda_{H,M} &= \mathcal{P}_* {\Lambda}_{\vert M},
\end{align}
where $\Lambda$ is the Jacobi structure associated to the contact form $\eta$. So, for $x \in M \subseteq P$ and $\alpha,\beta\in T_{x}^{*} P$, we have
$$\Lambda_{H,M}\left(\alpha,\beta\right) =  \Lambda\left(\mathcal{P}^*\left(\alpha\right),\mathcal{P}^*\left(\beta\right)\right).$$

\noindent{This structure provides an associated morphism of vector bundles}
\begin{equation}
  \begin{aligned}
    \sharp_{\Lambda_{H,M}}:  T^{*}_{M}P &\to T_{M}P,\\
    \alpha &\mapsto \Lambda_{H,M}(\alpha, \cdot)
  \end{aligned}
\end{equation}

\begin{lemma}
The following identity holds
$$\mathcal{Q}^{*}\left( \dd H  \right)= 0.$$

\begin{proof}
It follows directly from the mechanical condition.
\end{proof}
\end{lemma}
Then, by using the above Lemma, we have that,
  \begin{equation}\label{45.2}
    \mathcal{P}(\sharp_{\Lambda}(\dd H)) = \sharp_{\Lambda_{H,M}}(\dd H).
  \end{equation}
Thus, we have
  \begin{align*}
   X_{H, M} &= \mathcal{P}(X_{H}) \\ &=
    \mathcal{P}(\sharp_{\Lambda_H}(\dd H) -  H \Reeb) \\ &=
    {\sharp_{\Lambda_{H,M}}}(\dd H) - H \Reeb_{H,M},
  \end{align*}
along $M$, and, therefore, we obtain another explicit expression for the solution to \cref{6.4}, namely
 \begin{equation}\label{46.2}
    X_{H, M} =
    {\sharp_{\Lambda_{H,M}}}(\dd H) - H \Reeb_{H,M},
  \end{equation}

Furthermore, we can define the following bracket from functions on $P$ to functions on $M$, which will be called the \emph{nonholonomic bracket}:
\begin{equation}\label{gennonholonbracketdef56}
  \set{f,g}_{H,M} = \Lambda_{H,M}(\dd f, \dd g) - f \Reeb_{H,M}(g) + g \Reeb_{H,M}(f).
\end{equation}

\begin{theorem}\label{genesteteorema23}
  The nonholonomic bracket (\ref{gennonholonbracketdef56}) has the following properties:
  \begin{enumerate}
    \item Any function $g$ on $P$ that is constant on $M$ is a Casimir, i.e.,
    \[\set{f,g}_{H,M} = 0, \ \forall f \in \mathcal{C}^{\infty} \left( T^{*}Q \times \RR \right)\]
    \item The bracket provides the evolution of the observables, that is,
    \begin{equation}
      X_{H,M}(f) = \set{H,f}_{H,M} - f \Reeb_{H,M} (H).
    \end{equation}
  \end{enumerate}
\end{theorem}
\begin{proof}
For the first assertion, let $g$ be a function which is constant on $M$ and let $f$ be an arbitrary function on $P$. Notice that this implies that $\dd g \in \ann{(TM)}$, hence $\mathcal{P}^* (\dd g) = 0$. Then, along $M$, we have that
  \begin{align*}
    \set{g,f}_{H,M} &=
    \Lambda_{H,M}(\dd g,\dd f) - g \Reeb_{H,M} (f) + f \Reeb_{H,M} (g) \\ &=
    \Lambda_{Q}(\mathcal{P}^*(\dd g),\mathcal{P}^*(\dd f))  = 0.
  \end{align*}
On the other hand, by \cref{46.2} we deduce
\begin{equation*}
        \set{H,f}_{H,M} - f \Reeb_{H,M} (H) =
        \set{{\sharp_{\Lambda_{H,M}}}(\dd H)}\left(f\right) -  H \Reeb_{H,M}(f) =
        X_{H,M}(f),
\end{equation*}
\end{proof}
Notice that, in particular, all the functions $\varphi^{a}$ defining $M$ are Casimirs.\\
It is also remarkable that, using the statement \textit{1.} in \cref{genesteteorema23}, the nonholonomic bracket may be restricted to functions on $M$. \textit{Thus, from now on, we will refer to the nonholonomic bracket as the restriction of $\set{\cdot , \cdot }_{H,M}$ to functions on $M$.}
\begin{proposition}
The nonholonomic bracket endows the space of differentiable functions on $M$ with an almost Lie algebra structure which satisfies the generalized Leibniz rule
        \begin{equation}\label{geneq:mod_leibniz_rulenonhol}
           \set{f,gh}_{H, M} = g\set{f,h}_{H, M} + h\set{f,g}_{H, M} -  g h \Reeb_{H, M}(h).
        \end{equation}
\begin{proof}
Obviously, the nonholonomic bracket is $\RR-$linear and skew-symmetric. The generalized Leibniz rule follows from a straightforward computation.
\end{proof}

\end{proposition}

\begin{corollary}
The vector field $ \Reeb_{H,M}$ and the bivector field $\Lambda_{H,M} $ induce a Jacobi structure on $M$ if, and only if, the nonholonomic bracket satisfies the Jacobi identity.
\end{corollary}
These results motivate the following definition, introduced in \cite{primero}.
\begin{definition}
Let $M$ be a manifold with a vector field $E$ and a bivector field $\Lambda$. The triple $(M,\Lambda,E)$ is said to be an \textit{almost Jacobi structure} if the pair $(\Cont^\infty(M),\jacBr{\cdot,\cdot})$ is an almost Lie algebra satisfying the generalized Leibniz rule (\ref{eq:mod_leibniz_rule}) where the bracket is given by
\begin{equation}\label{genalmostjabracket243}
  \jacBr{f,g} = \Lambda(\dd f, \dd g) + f E(g) - g E(f)
\end{equation}
\end{definition}
With this definition, we may say that the triple $\left( M,  \Lambda_{H,M}, -\Reeb_{H,M}  \right)$ is an \textit{almost Jacobi structure}.

We will now present two applications of this general framework. The first one retrieves the constraint dynamics for Lagrangian systems studied in \cite{primero}. The second one, permits us to give the Hamiltonian counterpart on $T^{*}Q\times \RR$ from where, among other results, we we will be able to present the \textit{Hamiltonian Herglotz equations} \cref{eq:Hamiltonnonholonomic_herglotz_eqs_coords}.

\subsection{The P-nonholonomic bracket}
Notice that to construct the almost Jacobi structure $\left( M,  \Lambda_{H,M}, -\Reeb_{H,M}  \right)$, we project the Jacobi structure induced by the contact manifold $\left( P, \eta\right)$ by $\mathcal{P}$. The motivation behind this construction comes from the fact of that the projection of the Hamiltonian vector field $X_{H}$ by $\mathcal{P}$ is the constrained Hamiltonian vector field $X_{H,M}$.

Let $\left(P,  \eta, H , M , F \right)$ be a constrained Hamiltonian system satisfying the \textit{mechanical condition}, with $\Reeb$ tangent to $F$ and $\orthL{F}\subseteq F$. Then,
\begin{eqnarray*}
{\Psi}^{a}\left( X_{H,M}\right) &=&  {\Psi}^{a}\left(  \sharp\left(\dd H - \left(H + \Reeb\left(H\right)\right)\eta \right)_{\vert M}\right)\\
&=&{\Psi}^{a}\left(  \sharp\left(\dd H \right)_{\vert M}\right)\\
&=& 0
\end{eqnarray*}
The first equality is a consequence of the inclusion $\orthL{F}\subseteq F$, the second equality holds since $\Reeb \in F$, and the last equality is a consequence of the mechanical condition. Then, the mechanical condition implies that the constrained Hamiltonian vector field $X_{H,M}$ is tangent to $F$.

Let us now assume the conditions of uniqueness in Proposition \ref{GeneralizedYoootraproposiciondelcarajomas} and existence in Corollary \ref{equaldimensions34}. Then, we have that
\begin{equation}\label{neccforexisten}
\orthL{TM} \cap F \ = \ \set{0} ,  
\end{equation}
Therefore, $TM \cap F$ satisfies that
\begin{equation}\label{simplecsubs34}
\left( TM \cap F\right) \oplus \orthL{\left( TM \cap F\right)} = T_{M}P
\end{equation}
So, we obtain the projections,
\begin{subequations}\label{otromasprojectors2}
\begin{align}
\mathrm{P}: T_{M}P &\rightarrow TM \cap F,\\
\mathrm{Q} :  T_{M}P &\rightarrow \orthL{\left( TM \cap F\right)}.
\end{align}
\end{subequations}
Notice that $\mathcal{P} \left( {X_{H}}_{\vert M} \right) = X_{H , M} \in TM \cap F$ and $\mathcal{Q} \left( {X_{H}}_{\vert M} \right)  \in \orthL{F}  \subseteq \orthL{\left( TM \cap F\right)}$. Thus, by the uniqueness, we get
$$\mathrm{P} \left( {X_{H}}_{\vert M} \right) = \mathcal{P} \left( {X_{H}}_{\vert M} \right) = X_{H , M}$$
Therefore, $\mathrm{P}$ provides another projection to obtain the constrained Hamiltonian vector field associated to the constrained Hamiltonian system $\left(P, \eta, H, M, F\right)$ satisfying the mechanical condition in such a way that the Reeb vector field $\Reeb$ is tangent to $F$ and $\orthL{F}\subseteq F$.

Consider a (local) basis of $1-$forms $\set{\Psi^{a} \ =\ \Phi^{a}_{i}\dd q^{i} \ + \ \Psi^{a}_{i}\dd p_{i} \ + \ \mu^{a} \dd z} $ of $\ann{F}$. Then, for all $a,b$
\begin{equation*}
    \Psi^{b}\left(\sharp \left(\Psi^{a}\right)\right) = \left( \Phi^{b}_{i}\Psi^{a}_{i} -  \Phi^{a}_{i}\Psi^{b}_{i}\right) + \left( \Psi^{a}_{i}\mu^{b}p_{i} - \Psi^{b}_{i}\mu^{a}p_{i}\right) + \mu^{a}\mu^{b}
\end{equation*}
So, $\orthL{F}\subseteq F$ and $\Reeb \in F$ are satisfied if, and only if, we have that
\begin{align*}
    \mu^{a} &= 0\\
    \Phi^{b}_{i}\Psi^{a}_{i} -  \Phi^{a}_{i}\Psi^{b}_{i} &= 0,
\end{align*}
for all $a,b$.

Let $\varphi^{b}$ be constraint functions locally defining $M$. Then, $\orthL{TM}$ is generated by 
\begin{eqnarray*}
Y_{c} &=& \sharp \left( \dd \varphi^{c}\right)\\
&=& \frac{\partial \varphi^{c}}{\partial q^{i}}\sharp \left( \dd q^{i}\right) + \frac{\partial \varphi^{c}}{\partial p_{i}}\sharp \left( \dd p_{i}\right) +\frac{\partial \varphi^{c}}{\partial z}\sharp \left( \dd z\right)\\
&=& -\frac{\partial \varphi^{c}}{\partial q^{i}} \frac{\partial}{\partial p_{i}} + \frac{\partial \varphi^{c}}{\partial p_{i}} \left(  \frac{\partial}{\partial q_{i}} + p_{i} \frac{\partial}{\partial z} \right) +\frac{\partial \varphi^{c}}{\partial z} 
 \left( \frac{\partial}{\partial z} -  p_{i} \frac{\partial}{\partial p_{i}} \right)\\
&=&    \frac{\partial \varphi^{c}}{\partial p_{i}}  \frac{\partial}{\partial q_{i}}  - \left( \frac{\partial \varphi^{c}}{\partial q^{i}}+  p_{i}\frac{\partial \varphi^{c}}{\partial z}\right)\frac{\partial}{\partial p_{i}} +\left(\frac{\partial \varphi^{c}}{\partial z}  + p_{i} \frac{\partial \varphi^{c}}{\partial p_{i}}\right) \frac{\partial}{\partial z}
\end{eqnarray*}

So, $\orthL{\left(TM \cap F\right)}$ is locally generated by the family of (local) vector fields $\set{Y_{c},X_{a}}$, where,
\begin{equation*}
X_{a} = \sharp \left( {\Psi}^{a}\right) = \Psi^{a}_{i}\frac{\partial}{\partial q^{i}} - \left[\Phi^{a}_{i} + \mu^{a}p_{i}\right] \frac{\partial}{\partial p_{i}} + \left[\mu^{a}+\Psi^{a}_{k}p_{k}\right]\frac{\partial}{\partial z}
\end{equation*}
is a (local) basis of sections of the subbundle $\orthL{F}$. Consider a (local) basis $\{ \theta_{i}\}$ of $T M \cap F$. Then, the restriction of any (local) vector field $Y \in \mathfrak{X}\left(P\right)$ to $M$ may be locally written as,
\[ Y_{\vert M} = R^{i}\theta_{i} + \nu^{c}Y_{c} + \lambda^{a}X_{a}.\]
Then, by applying $\Psi^{d}$ ($\orthL{F} \subseteq F$), we get
$$\Psi^{d} \left( Y \right) = \nu^{c} \Psi^{d} \left(Y_{c}\right)$$
Let us denote by $\mathcal{G}$ to the matrix with coefficients,
$$ \mathcal{G}_{dc}= \Psi^{d} \left(Y_{c}\right).$$
Hence,
$$\nu^{c} = \Psi^{d} \left( Y \right)\mathcal{G}^{dc},$$
where $\mathcal{G}^{dc}$ are the coefficients of the inverse matrix of $\mathcal{G}$. On the other hand, applying $\dd \varphi^{b}$ we have that
\[ \dd \varphi^{b} \left( Y \right) = Y \left( \varphi^{b}\right) = \nu^{c}Y_{c}\left( \varphi^{b}\right) + \lambda^{a} \mathcal{C}_{ba},\]
where $\mathcal{C}$ is the matrix with coefficients given by \cref{matrixcoeff34}. Therefore,
$$\lambda^{a} = \left[Y \left( \varphi^{b}\right) -\nu^{c}Y_{c}\left( \varphi^{b}\right)\right] \mathcal{C}^{ba} = \left[Y \left( \varphi^{b}\right) -\Psi^{d} \left( Y \right)\mathcal{G}^{dc}Y_{c}\left( \varphi^{b}\right)\right] \mathcal{C}^{ba} ,$$
where $\mathcal{C}^{ba}$ are the coefficients of the inverse matrix of $\mathcal{C}$.
Hence,
\begin{align*}
\mathrm{Q} \left( Y_{\vert M} \right) &=   \Psi^{d} \left( Y \right)\mathcal{G}^{dc}Y_{c} + \left[Y \left( \varphi^{b}\right) -\Psi^{d} \left( Y \right)\mathcal{G}^{dc}Y_{c}\left( \varphi^{b}\right)\right] \mathcal{C}^{ba}X_{a}.\\
\mathrm{P} \left( Y_{\vert M} \right) &= Y_{\vert M} -  \Psi^{d} \left( Y \right)\mathcal{G}^{dc}Y_{c} - \left[Y \left( \varphi^{b}\right) -\Psi^{d} \left( Y \right)\mathcal{G}^{dc}sY_{c}\left( \varphi^{b}\right)\right] \mathcal{C}^{ba}X_{a}.
\end{align*}
Therefore, an \textit{explicit (local) expression} of the solution $X_{H, M}$ is given by
\begin{equation}\label{20.22}
X_{H, M} = \left(X_{H}\right)_{\vert M} - \Psi^{d} \left( X_{H} \right)\mathcal{G}^{dc}Y_{c} - \left[X_{H}\left( \varphi^{b}\right) -\Psi^{d} \left( Y \right)\mathcal{G}^{dc}Y_{c}\left( \varphi^{b}\right)\right] \mathcal{C}^{ba}X_{a}
\end{equation}
Thus, we have another local expression for the constrained Hamiltonian vector field $X_{H, M}$.

So, it makes sense to define the following almost Jacobi structure,
\begin{align}\label{anothermorealmostjac24}
    \overline{\Reeb}_{H,M} &=  \mathrm{P} \left({\Reeb}_{\vert M}\right),\\
    \overline{\Lambda}_{H,M} &= \mathrm{P}_* {\Lambda}_{\vert  M},
\end{align}
Then, we may define the following bracket of functions on $M$
\begin{equation}
  \set{f,g}^{\mathrm{P}}_{H,M} = \overline{\Lambda}_{H,M}(\dd f, \dd g) - f \overline{\Reeb}_{H,M}(g) + g \overline{\Reeb}_{H,M}(f).
\end{equation}
Let us call this bracket as \textit{$\mathrm{P}-$nonholonomic bracket}. 

Notice that, as a consequence of $\orthL{F}\subseteq F$, we may conclude that any vector $Y$ tangent to $F$ satisfies that
$$ \mathcal{P}\left( Y \right) \ = \ \mathrm{P}\left( Y \right)$$
Therefore, $\overline{\Reeb}_{H,M} = \Reeb_{H,M}$.

On the other hand, we may consider the adjoint operators 
\begin{align*}
    \mathcal{P}^* :T^{*}_{M} P &\rightarrow \ann{\left(\orthL{F}\right)}\\
    \mathcal{Q}^* :T^{*}_{M} P&\rightarrow \ann{T M}\\
    \mathrm{P}^* :T^{*}_{M} P &\rightarrow \ann{\left(\orthL{\left(F\cap T M\right)}\right)}\\
    \mathrm{Q}^* :T^{*}_{M} P &\rightarrow \ann{\left(F\cap T M\right)}
\end{align*}
Then, applying $\sharp$ to $\ann{\orthL{F}}$ and $\ann{T M}$, we obtain a new decomposition of $T^{*}_{M} P$,
$$T^{*}_{M} P = F \oplus \orthL{T M}$$
with associated projections, $\mathbb{P} :T^{*}_{M} P \rightarrow F$ and $\mathbb{Q} :T^{*}_{M} P \rightarrow \orthL{T M}$ given by
$$ \mathbb{P} = \sharp\circ \mathcal{P}^{*}\circ \flat , \ \ \ \ \ \mathbb{Q} = \sharp \circ \mathcal{Q}^{*}\circ \flat$$
Performing the same operation with the projections $\mathrm{P}^*$ and $\mathrm{Q}^*$, we obtain the splitting 
$$T_{M}P =\left( TM \cap F\right) \oplus \orthL{\left( TM \cap F\right)} $$
Hence, by uniqueness, we deduce that
\begin{equation}
    \sharp\circ \mathrm{P}^{*}\circ \flat = \mathrm{P}, \ \ \ \ \ \  \sharp\circ \mathrm{Q}^{*}\circ \flat = \mathrm{Q}
\end{equation}
Let us now consider $V,W \in T_{M}P$. Then,
\begin{eqnarray*}
    \Lambda_{H,M}\left(\flat\left( V \right),\flat\left( W \right)\right) &=&  \Lambda \left(\mathcal{P}^*\left(\flat\left( V \right)\right),\mathcal{P}^*\left(\flat\left( W\right)\right)\right)\\
&=& - \dd \eta \left( \sharp\left(\mathcal{P}^*\left(\flat\left( V \right)\right)\right),\sharp\left(\mathcal{P}^*\left(\flat\left( W\right)\right)\right)\right)\\
    &=& - \dd \eta\left(\mathbb{P}\left(  V \right), \mathbb{P}\left(  W\right)\right)
\end{eqnarray*}
Notice that, $\Reeb \in F$ implies that $\eta\left( \orthL{F}\right) =0$, and therefore,
\begin{eqnarray*}
    \dd \eta_{Q}\left(\mathrm{Q}\left(\mathbb{P}\left(  V \right)\right), \mathbb{P}\left(  W\right)\right) &=& \set{\flat_{Q}\left(\mathrm{Q}\left(\mathbb{P}\left(  V \right)\right)\right) } \left(\mathbb{P}\left(  W\right)\right)\\
    &=& 0
\end{eqnarray*}
The last equality is consequence of that $\mathbb{P}\left(  W\right) \in F$ and $\flat\left(\mathrm{Q}\left(\mathbb{P}\left(  V \right)\right)\right) \in F$. Hence,
\begin{eqnarray*}
    \Lambda_{H,M}\left(\flat\left( V \right),\flat\left( W \right)\right)  &=& - \dd \eta\left(\mathrm{P}\left(\mathbb{P}\left(  V \right)\right), \mathbb{P}\left(  W\right)\right)\\
    &=& - \dd \eta \left(\mathrm{P}\left(\mathbb{P}\left(  V \right)\right), \mathrm{P}\left(\mathbb{P}\left(  W\right)\right)\right)
\end{eqnarray*}
Observe that
$$ V = \mathrm{P}\left(\mathbb{P}\left(  V \right)\right) + \mathrm{Q}\left(\mathbb{P}\left(  V \right)\right) + \mathbb{Q}\left(  V \right)$$
Nevertheless, $\mathrm{Q}\left(\mathbb{P}\left(  V \right)\right) \in  \orthL{F} \subseteq F$ and $\mathbb{Q}\left(  V \right) \in \orthL{TM}$. Then,
$$\mathrm{Q}\left(\mathbb{P}\left(  V \right)\right) + \mathbb{Q}\left(  V \right) \in \orthL{\left[F\cap TM\right]}$$
Therefore, we get
$$\mathrm{P}\left(\mathbb{P}\left(  V \right)\right) = \mathrm{P}\left( V \right)$$
In other words, we obtain
\begin{eqnarray*}
    \Lambda_{H,\mathcal{M}}\left(\flat\left( V \right),\flat\left( W \right)\right)  &=&  - \dd \eta\left(\mathrm{P}\left(\mathbb{P}\left(  V \right)\right), \mathrm{P}\left(\mathbb{P}\left(  W\right)\right)\right)\\
   &=&  - \dd \eta\left(\mathrm{P}\left(  V \right), \mathrm{P}\left( W\right)\right)\\
   &=&  \overline{\Lambda}_{H,\mathcal{M}}\left(\flat\left( V \right),\flat\left( W \right)\right) 
\end{eqnarray*}
Thus, we have proved the following theorem.
\begin{theorem}
The nonholonomic bracket $\set{\cdot, \cdot}_{H,M}$ and the $\mathrm{P}-$nonholonomic bracket $\set{\cdot, \cdot}^{\mathrm{P}}_{H,M}$ coincide.    
\end{theorem}

Notice that both the Lagrangian and the Hamiltonian descriptions satisfy the mechanical condition, the Reeb vector field is tangent to $F$, and $\orthL{F}\subseteq F$ (see the next two sections for a proof). Therefore, this equivalent version of the non-holonomic bracket may be defined in both cases, giving a \textit{completely new way} (in both cases) of calculating the bracket of function in a (Hamiltonian or Lagrangian) non-holonomic contact system.

\section{Contact Lagrangian dynamics}
A particular but relevant case of contact manifold is given by the manifold $TQ \times \mathbb{R}$ where $Q$ is an $n-$dimensional manifold, equipped with a \textit{regular Lagrangian function} $L:TQ\times \RR \to \RR$, i.e., a function on $TQ\times \RR$ satisfying that the matrix,
\begin{equation}\label{eq:velocity_Hessian}
	W_{ij}=\left( \frac{\partial^2 L}{\partial \dot{q}^i \partial \dot{q}^j} \right)
\end{equation}
is regular. So, we may construct a contact form on $TQ\times \RR$ as follows
\begin{equation}
    \eta_L = \dd z - \alpha_L,
\end{equation}
where $\alpha_L = S^* (\dd L)$, and $S$ is the \textit{canonical vertical endomorphism on} $TQ$ extended in the natural way to $TQ \times \RR$ (see~\cite{deLeon2011}). $\eta_{L}$ is called the \textit{contact Lagrangian form}. Let $\left( q^{i}, \dot{q}^i,z\right)$ the natural coordinates on $TQ\times \RR$. Then, 
\begin{align}
    \eta_L &=\dd z -  \pdv{L}{\dot q^i} \dd {q}^i,
\end{align}
Moreover, the energy of the system is defined as follows
 \begin{equation}
    E_L = \Delta(L) - L = \dot{q}^i \pdv{L}{\dot{q}^i} - L,
 \end{equation}
 where $\displaystyle{\Delta = \dot{q}^i \pdv{}{\dot{q}^i}}$ is the \textit{Liouville vector field on} $TQ$ extended in a natural way to $TQ\times \RR$. The triple $\left(TQ \times \RR, \eta_{L} ,  E_{L} \right)$ is called a \textit{contact Lagrangian system}.

The dynamics is given by the \emph{Euler--Lagrangian vector field}, which is $\Gamma_L = X_{E_{L}}$ (Eq. (\ref{eq:hamiltonian_vf_contact})) whose integral curves $\left( \xi , \dot{\xi} , s \right)$ are solution of the following equations
\begin{eqnarray}\label{52}
   && \pdv{L}{q^i} - \dv{}{t} \pdv{L}{\dot{q}^i} + \pdv{L}{\dot{q}^i} \pdv{L}{z} = 0\\
   && \dot{z} = L
\end{eqnarray}

These equations were presented by G. Herglotz in 1930 \cite{Herglotz1930} (see also \cite{Georgieva2003,Georgieva2011}), and they will be called \textit{Herglotz equations}. In a more recent approach, A. Bravetti \emph{et al.}~\cite{Bravetti2018a} and M. de Le\'on and M. Lainz ~\cite{deLeon2019} have obtained the connection between these equations and the contact Hamiltonian dynamics. Furthermore, we may derive these Herglotz equations from the \textit{Herglotz variational principle}~\cite[Thm.~2]{deLeon2019}.\\

In \cite{primero}, the non-holonomic equations for contact Lagragian systems with linear kinematic constraints were studied, by restricting the variations so that they satisfy the nonholonomic constraints in the Herglotz variational principle. We will prove that this may be understood as a particular case of our general framework of constrained dynamics.

Let us assume that we have linear kinematic constraints (that is, constraints linear on the velocities) defined by a regular distribution $\mathcal{D}$ on the configuration manifold $Q$ with codimension $k$. Then, there exists a family of local $1-$forms $\{\Phi^a \}_{a=1, \dots , k}$ on $Q$ such that
\begin{equation} \label{5}
\mathcal{D}=\left\{v\in TQ \mid \Phi^a \left( v \right)=0\right\}.
\end{equation}
Consider the Lagrangian function $L: TQ \times \mathbb{R} \rightarrow \mathbb{R}$, and define the set
\begin{equation}
    \Omega(q_1,q_2, [a,b])_{\xi}^{\mathcal{D}} = \set{
        v \in T_\xi \Omega(q_1,q_2, [a,b]) \mid v(t) \in \mathcal{D}_{\xi(t)} \text{ for all } t \in [a,b]},
\end{equation}
where, $\Omega(q_1,q_2, [a,b]) \subseteq(\Cont^\infty([a,b]\to Q)$ is the infinite dimensional manifold given by the smooth curves $\xi$ on an interval $[a,b] \subset \RR$ such that $\xi(a)=q_1$ and $\xi(b)=q_2$. Then,
\begin{equation*}
\begin{aligned}
        T_\xi \Omega(q_1,q_2, [a,b]) =  \set{&
            v_\xi \in \Cont^\infty([a,b] \to TQ) \mid \\&
            \tau_Q \circ v_\xi = \xi, \,v_\xi(a)=0, \, v_\xi(b)=0
            }.
\end{aligned}
\end{equation*}
So, $\xi$ satisfies the \textit{Herglotz variational principle with constraints} \cite{primero} if, and only if,
\begin{enumerate}
\item $T_{\xi}\mathcal{A}_{\vert \Omega(q_1,q_2, [a,b])_{\xi}^{\mathcal{D}}} = 0.$

\item $\dot{\xi} \left( t \right) \in \mathcal{D}_{\xi(t)} \text{ for all } t \in [a,b].$
\end{enumerate}

\begin{theorem}[\cite{primero}]
A path $\xi \in \Omega(q_1,q_2, [a,b])$ satisfies the Herglotz variational principle with constraints if and only if
\begin{equation}
    \begin{dcases}
\left(\pdv{L}{q^i} - \dv{}{t} \pdv{L}{\dot{q}^i} + \pdv{L}{\dot{q}^i} \pdv{L}{z}\right)dq^{i}  \in \ann{\mathcal{D}}_{\xi(t)}
\\
\dot{\xi} \in \mathcal{D}
\end{dcases}
\end{equation}
where $\ann{\mathcal{D}} = \set {a \in T^*Q \mid a(u)=0 \text{ for all } u \in \mathcal{D} }$ is the annihilator of $\mathcal{D}$.
\end{theorem}
In other words, there exist some Lagrange multipliers $\lambda_i(q^i)$ such that,
\begin{equation}
    \begin{dcases}\label{eq:nonholonomic_herglotz_eqs_coords}
        \dv{}{t} \pdv{L}{\dot{q}^i} - \pdv{L}{q^i} - \pdv{L}{\dot{q}^i} \pdv{L}{z} = \lambda_a\Phi^a_i \\
        \Phi^a (\dot{\xi}(t)) = 0.
    \end{dcases}
\end{equation}
From now on, Eqs. (\ref{eq:nonholonomic_herglotz_eqs_coords}) will be called the \emph{constrained Herglotz equations}.

The distribution $\mathcal{D}$ may be naturally lifted into a distribution $\mathcal{D}^{l}$ on $TQ\times \mathbb{R}$ in such a way that
\begin{equation}\label{eqsDo}
    \ann{(\mathcal{D}^{l})} = \left( \tau_{Q,\mathbb{R}}\right)^{*} \ann{\mathcal{D}}= S^{*} \left( \ann{T \left(\mathcal{D} \times \mathbb{R}\right)}\right) = <\tilde{\Phi}^{a}= \Phi^{a}_{i}{\dd q}^{i}>,
\end{equation}
where  $\tau_{Q,\mathbb{R}}: TQ \times \mathbb{R} \rightarrow Q$ is the projection given by $\tau_{Q,\mathbb{R}}\left( v_{q},z\right) = q$.
\begin{theorem}[\cite{primero}]\label{18}
Assume that $L$ is regular. Let $X$ be a vector field on $TQ \times \RR$ satisfying the equation
\begin{equation}
    \begin{dcases}\label{6}
        \flat_L\left(X\right) - \dd E_L + \left(E_L + \Reeb_L\left(E_L\right)\right)\eta_L \in \ann{\mathcal{D}^{l}} \\
        X_{\vert \mathcal{D} \times \RR} \in  \mathfrak{X}\left(\mathcal{D} \times \RR \right).
    \end{dcases},
\end{equation}
Then,
\begin{itemize}
\item[(1)] $X$ is a SODE on $TQ \times \RR$.

\item[(2)] The integral curves of $X$ are solutions of the constrained Herglotz equations (\ref{eq:nonholonomic_herglotz_eqs_coords}).
\end{itemize}

\end{theorem}

\begin{remark}
{\rm 
Recall that a \textit{SODE} is a vector field $X$ on $TQ \times \mathbb{R}$ such that,
\begin{equation}\label{55}
S \left( X \right) = \Delta
\end{equation}
Let $\left( q^{i} \right) $ be a local chart on $Q$. Then, Eq. (\ref{55}) reduces to
\[ X \left( q^{i} \right) = \dot{q}^{i}, \ \forall i.\]
Thus, $X$ is a SODE if and only if any integral curve of $X$ is written (locally) as $\left( \xi , \dot{\xi} , s \right)$ for some (local) curve $\xi$ on $Q$.
Observe that, it is immediate to prove that the vector field $X$ defined above is a SODE.}
\end{remark}

By considering $M = \mathcal{D}\times \mathbb{R}$ and $F = \mathcal{D}^{l}$, we have that \cref{6} reduces to \cref{6.4} and the quintuple $\left(TQ \times \mathbb{R}, \eta_L, E_{L},  \mathcal{D}\times \mathbb{R} , \mathcal{D}^{l} \right)$ is a constrained Hamiltonian system satisfying the \textit{mechanical condition} (Definition \ref{mechanicalcondit2343}).

Indeed, the local basis of $\ann{F}$ is given by the $1-$forms $\tilde{\Phi}^{a}= \Phi^{a}_{i}{\dd q}^{i}$, and it is trivial to check the mechanical condition and $\orthL{F}\subseteq F$. Furthermore, $\Reeb_{L}\left(\tilde{\Phi}^{a} \right) = 0$ for all $a$, i.e., $F$ is vertical.

Furthermore, at it is proved in \cite{primero}, the condition of existence in Corollary \ref{equaldimensions34} is always satisfied, and the condition of uniqueness of Proposition \ref{GeneralizedYoootraproposiciondelcarajomas} is satisfied, if and only if, the matrix,
\begin{equation}\label{lamatrizcdecarajo}
\left(\mathcal{C}_{ab}\right) =  \left(W^{ik}\Phi^{b}_{k}\Phi^{a}_{i} \right),
\end{equation}
where $\left( W^{ik} \right)$ is the inverse of the Hessian matrix $\left( W_{ik}\right) = \left(\dfrac{\partial^{2}L}{\partial\dot{q}^i \partial \dot{q}^{k}}\right)$, is regular.

It is a simple exercise to check that the \textit{non-holonomic bracket} for the constrained Hamiltonian system $\left(TQ \times \mathbb{R}, \eta_L, E_{L},  \mathcal{D}\times \mathbb{R} , \mathcal{D}^{l} \right)$ 
given in \cref{gennonholonbracketdef56}, coincides with the one given in \cite{primero}.

\section{Constrained dymanics of contact Hamiltonian systems}
In this section, we will use the general framework of constrained dynamics developed in Section \ref{genframework24}, to obtain the Hamiltonian counterpart of non-holonomic systems.

Here, we will consider a \textit{Hamiltonian function} as a differentiable function $H : T^{*}Q\times \mathbb{R} \rightarrow \mathbb{R}$. We may define the canonical contact form of $T^{*}Q\times \mathbb{R}$ given by
\begin{equation}
    \eta_{Q} = dz - \theta_{Q},
\end{equation}
where $\theta_{Q}$ is the \textit{canonical Liouville form on} $T^{*}Q\times \mathbb{R}$. Namely,
$$\set{\theta_{Q} \left( \alpha_{q} , z\right)}\left( X_{\alpha_{q}}, \lambda_{z}\right) = \alpha_{q}\left(T_{\alpha_{q}}\pi_{Q,\mathbb{R}}\left( X_{\alpha_{q}}\right) \right)$$
for all $\left( \alpha_{q} , z\right)\in T^{*}Q\times \mathbb{R}$, and $\left( X_{\alpha_{q}}, \lambda_{z}\right) \in T_{\left( \alpha_{q} , z\right)}\left(T^{*}Q\times \mathbb{R} \right) $, where $\pi_{Q,\mathbb{R}}: T^{*}Q \times \mathbb{R} \rightarrow Q$ is the projection defined by $\pi_{Q,\mathbb{R}}\left(\alpha_{q}, z\right) = q$.

Thus, considering the canonical coordinates $\left( q^{i}, p_{i} , z \right)$, we have that
\begin{equation}
    \eta_{Q} = dz - p_{i} \dd q^{i},
\end{equation}
Then, the Hamiltonian vector associated to $H$ is locally described as follows,
\begin{equation}\label{Hamiltneq2}
    X_{H} =\pdv{H}{p_i}\pdv{}{q^i}  -\left(\pdv{H}{q^i}  +p_{i} \pdv{H}{z}\right) \pdv{}{p_{i}}    + \left( p_{i} \pdv{H}{p_{i}}-H\right)\pdv{}{z}
\end{equation}
Therefore, its integral curves satisfy the \textit{contact Hamiltonian equations}
\begin{equation}\label{Hamiltoneqs34}
\begin{dcases}
 \frac{\dd q^{i}}{\dd t} &=    \pdv{H}{p^{i}}\\
 \frac{\dd p_{i}}{\dd t} &=  - \pdv{H}{q^{i}}  - p_{i} \pdv{H}{z}\\
 \frac{\dd z}{\dd t} &=   p_{i} \pdv{H}{p_{i}}-H
 \end{dcases}
\end{equation}

From a physical intuitive view, we will derive the Hamiltonian constrained equations from the Lagrangian constrained equations via the Legendre transformation. The \textit{Legendre transformation}, is defined as follows,
$$
\begin{array}{rccl}
FL : & TQ \times \mathbb{R} & \rightarrow & T^{*}Q \times \mathbb{R}\\
&\left( v_{q} , z \right) &\mapsto & \dd \left(L_{T_{q}Q \times \set{z} }\right)_{\vert \left( v_{q} , z \right)}\\
\end{array}
$$
i.e., locally
$$FL \left( q^{i}, \dot{q}^{i},z\right) = \left( q^{i}, \pdv{L}{\dot{q}^i},z\right)$$
Along this paper, we will assume that $FL$ is a global diffeomorphism, namely, $L$ is \textit{hyperregular}. Notice that, in this case, the \textit{Hamiltonian energy} is defined by
\begin{equation}\label{Hamiltonianfunction}
    H = E_{L}\circ \left( FL \right)^{-1}
\end{equation}
Then,
\begin{equation}\label{eqsmuchas3434}
    \begin{dcases}
FL^{*}\eta_{Q} &= \eta_{L}\\
FL_{*} \Reeb_{L} &= \Reeb_{Q}\\
FL_{*}\Gamma_{L} &= X_{H}\\
T^{*}FL\circ {\flat}_{Q} \circ T FL &= {\flat}_{L}
\end{dcases}
\end{equation}

Hence, $FL$ transforms the Herglotz equations (\ref{52}) into the contact Hamiltonian equations (\ref{Hamiltoneqs34}).

Let us now assume the existence of linear kinematic constraints (on the velocities) defined by a regular distribution $\mathcal{D}$ on the configuration manifold $Q$ with codimension $k$ and a family of local $1-$forms $\{\Phi^a \}_{a=1, \dots , k}$ on $Q$ generating $\ann{\mathcal{D}}$ (see Eq. (\ref{5})). Let us denote by $M$ the image of the constraint manifold $\mathcal{D}\times \mathbb{R}$ by the Legendre transformation $FL$, say $M = FL(\mathcal{D} \times \mathbb{R})$.

Assume that there exists a solution $\Gamma_{L, \mathcal{D}}$ of the Eq. (\ref{6}). Then, we may define a vector field $X_{H,M} = FL_{*}(\Gamma_{L,\mathcal{D}}) $ on $T^{*}Q\times \mathbb{R}$. So, by taking into account Eqs. (\ref{eqsmuchas3434}), $X_{H,M}$ satisfies,
\begin{equation}\label{aux-43}
 \begin{dcases}
        FL^{*}\left[\flat_Q\left(X_{H, M}\right) - \dd H + \left(H + \Reeb_Q\left(H\right)\right)\eta_Q\right] \in \ann{\mathcal{D}^{l}} \\
        \left(X_{H,M}\right)_{\vert M} \in  \mathfrak{X}\left(M \right).
    \end{dcases},
\end{equation}
Notice that, by Eq. (\ref{eqsDo}), we have
\begin{eqnarray}
{FL^{-1}}^{*} \left( \ann{\mathcal{D}^{l}} \right) &=& {FL^{-1}}^{*} \left( \left( \tau_{Q,\mathbb{R}}\right)^{*} \ann{\mathcal{D}} \right)\\
 &=& \left( \tau_{Q,\mathbb{R}} \circ FL^{-1}\right)^{*} \ann{\mathcal{D}}\\
 &=&  \left( \pi_{Q,\mathbb{R}}\right)^{*} \ann{\mathcal{D}}
\end{eqnarray}
We will denote by $\mathfrak{D}^{l}$ the vector subbundle of $T^{*}Q \times \mathbb{R}$ such that $\left( \pi_{Q,\mathbb{R}}\right)^{*} \ann{\mathcal{D}} = \ann{\mathfrak{D}^{l}}$. Then, the annihilator $\ann{\mathfrak{D}^{l}}$ is also generated by the $1-$forms $\Phi^a_{i}dq^{i}$ on $T^{*}Q \times \mathbb{R}$ (see Eq. (\ref{5})) and, we may rewrite  Eq. (\ref{aux-43}) as follows
\begin{equation}
 \begin{dcases}\label{6.2}
        \flat_Q\left(X_{H, M}\right) - \dd H + \left(H + \Reeb_Q\left(H\right)\right)\eta_Q \in \ann{\mathfrak{D}^{l}} \\
        X_{\vert M} \in  \mathfrak{X}\left(M \right).
    \end{dcases},
\end{equation}

So, for the Hamiltonian counterpart we have the constrained Hamiltonian system given by $\left( T^{*}Q\times \RR , \eta_Q, H, FL\left( \mathcal{D}\times \RR\right) , \mathfrak{D}^{l}\right)$ (see Definition \ref{definitionconstraind23}). Therefore, locally, the constrained Hamiltonian vector field is given by,
\begin{equation}\label{nonholHamiltneq}
    X_{H,M} =\pdv{H}{p_i}\pdv{}{q^i}  -\left(\pdv{H}{q^Pi}  +p_{i} \pdv{H}{z}+ \lambda_a\Phi^a_i\right) \pdv{}{p_{i}}    + \left( p_{i} \pdv{H}{p_{i}}-H\right)\pdv{}{z}
\end{equation}

and its integral curves $\xi$ satisfy the following equations,
\begin{align}
    \begin{cases}\label{eq:Hamiltonnonholonomic_herglotz_eqs_coords}
         \frac{\dd q^{i}}{\dd t} &=   \pdv{H}{p^{i}}\\
 \frac{\dd p_{i}}{\dd t} &=  - \pdv{H}{q^{i}}  -p_{i} \pdv{H}{z} - \lambda_a\Phi^a_i\\
 \frac{\dd z}{\dd t} &=   p_{i} \pdv{H}{p_{i}}-H\\
        \varphi^{a}\left(\xi\right) & = 0, \ \forall a
    \end{cases}
\end{align}
where $\varphi^{a}$ are the constraint functions defined by $\Phi_a = \phi_a \circ FL$. From now on, Eqs. (\ref{eq:Hamiltonnonholonomic_herglotz_eqs_coords}) will be called the \emph{constrained Hamiltonian Herglotz equations}. Thus, we have proved the following result.

\begin{theorem}\label{18.2}
Let be a Hamiltonian function $H: T^{*}Q\times \mathbb{R} \rightarrow \mathbb{R}$, and a constraint manifold $M \subseteq T^{*}Q \times \mathbb{R}$. Let $X$ be a vector field on $T^{*}Q \times \RR$ satisfying the equation
\begin{equation*}
 \begin{dcases}
        \flat_Q\left(X\right) - \dd H + \left(H + \Reeb_Q\left(H\right)\right)\eta_Q \in \ann{\mathfrak{D}^{l}} \\
        X_{\vert M} \in  \mathfrak{X}\left(M \right)
    \end{dcases}
\end{equation*}
Then, the integral curves of $X$ are solutions of the constrained Hamiltonian Herglotz equations (\ref{eq:Hamiltonnonholonomic_herglotz_eqs_coords}). A solution for this equation will be denoted by $X_{H,M}$, and will be called \textit{constrained Hamiltonian vector field}.

\end{theorem}

Notice that, in general, $FL$ preserves the fibres of $\tau_{Q,\mathbb{R}}$ and $\pi_{Q,\mathbb{R}}$, but it is not linear on these fibres. Thus, we may not assume that $M$ is ``\textit{linear}'', i.e., a subbundle of $\pi_{Q,\mathbb{R}}$. Indeed, the momentum coordinates of the Legendre transformation $FL$ depend on the real variable ``$z$'', so we may not even assume that $M = \mathcal{M} \times \mathbb{R}$, for some subset $\mathcal{M}$ of $T^{*}Q$.

Nevertheless, there is a relevant family of Lagrangian functions which transform the space $\mathcal{D}\times \mathbb{R}$, with $\mathcal{D}$ a regular distribution of $Q$, into the space $\mathcal{M}\times \mathbb{R}$, with $\mathcal{M}$ a subbundle of $T^{*}Q$; these are those of \textit{mechanical type}. A lagrangian $L$ is of \textit{mechanical type} if there exists a Riemannian metric $g$ on $Q$, such that,
$$L \left(v_{q},z\right) \ = \ \frac{1}{2}g\left( v_{q},v_{q}\right) + V\left( q,z\right), \ \forall \left(v_{q},z\right) \in T_{q}Q\times \mathbb{R}.$$
Here, $T \left(v_{q}\right)= \frac{1}{2}g\left( v_{q},v_{q}\right)$ represent the \textit{kinetic energy}, and $V$ is the \textit{potential energy}. Then,
$$FL \left( v_{q},z\right) = \left(\flat_{g}\left( v_{q}\right),z\right),$$
where $\flat_{g}$ is the \textit{musical isomorphism} associated to $g$. Thus, in this case, we have
$$FL \left( \mathcal{D}\times \mathbb{R}\right) = \flat_{g}\left( \mathcal{D}\right) \times \mathbb{R},$$
and $\mathcal{M} = \flat_{g}\left( \mathcal{D}\right) = \ann{\mathcal{D}^{{\perp}_{g}}}$ is a vector subbundle of $T^{*}Q$ isomorphic to $\mathcal{D}$, where $\cdot^{\perp_{g}}$ defines the orthogonal complement respect to the Riemannian metric $g$. We will denote the inverse of $\flat_{g}^{-1}$ by $\sharp_{g}$.

\begin{remark}{\rm
For an arbitrary vector subbundle $\mathcal{M}$ of $T^{*}Q$, we may define the annihilator
$$ \ann{\mathcal{M}} := \set{v\in TQ  \mid  \alpha \left( v \right) = 0, \ \forall \alpha \in \mathcal{M}}$$
Then, $\ann{\mathcal{M}}$ is generated by local vector fields $Z_{a}$. So, in the Hamiltonian side of the picture, the local vector fields $Z_{a}$ will fulfill the role of the \textit{constraints}. Indeed,
$$\flat_{g}\left( \ann{\mathcal{M}}\right) \ = \ \ann{\left(\sharp_{g}\left(\mathcal{M}\right)\right)}$$
which means that the $1-$forms,
$$ \flat_{g}\left( Z_{a} \right) \ = \ \Phi^{a},$$
generates $\ann{\mathcal{D}}= \ann{\left(\sharp_{g}\left(\mathcal{M}\right)\right)}$. Notice that any vector field $X$ on $Q$ may be equivalently depicted as a function $X:T^{*}Q \rightarrow \mathbb{R}$ which is linear on the fibres of $T^{*}Q$. So, the \textit{linear constraints} of the Hamiltonian system will be given by the vector fields $Z_{a}$, as maps from $T^{*}Q$ to $\mathbb{R}$.
}
\end{remark}

Therefore, we may rewrite the constrained Hamiltonian equations (\ref{eq:Hamiltonnonholonomic_herglotz_eqs_coords}) in terms of a subbundle $\mathcal{M}$ of $T^{*}Q$ of codimension $k$. Notice that, using \cref{Hamiltonianfunction}, we have that our Hamiltonian function is given by 
\begin{equation}
H \left( \alpha_{q} , z  \right)= T \left( \sharp_{g}\left( \alpha_{q}\right) , \sharp_{g}\left( \alpha_{q}\right)\right) - V\left( q , z\right)  = \frac{1}{2}\left( \alpha_{q}\left( \sharp_{g}\left( \alpha_{q}\right) \right)\right) - V\left( q , z\right)    
\end{equation}
Locally,

\begin{equation}
    H \left( q^{i} , p_{i} , z  \right) = \frac{1}{2} g^{i,j}p_{i}p_{j} - V\left( q^{i} , z\right) ,
\end{equation}
where $g^{i,j}$ are the coefficients of the inverse matrix of $\left( g_{i,j} = g\left( \frac{\partial}{\partial q^{i}} , \frac{\partial}{\partial q^{j}}\right)\right)$. In fact, we may prove the following result for the case in which the constraint are given by a subbundle $\mathcal{M}$ of $T^{*}Q$.

\begin{theorem}\label{18.3}
Let be a Hamiltonian function $H: T^{*}Q\times \mathbb{R} \rightarrow \mathbb{R}$, and a constraint manifold $\mathcal{M} \subseteq T^{*}Q $ which is a subbundle of $T^{*}Q$. Let $X$ be a vector field on $T^{*}Q \times \RR$ satisfying the equations
\begin{equation}\label{6.3}
 \begin{dcases}
        \flat_Q\left(X\right) - \dd H + \left(H + \Reeb_Q\left(H\right)\right)\eta_Q \in \ann{\mathfrak{D}^{l}} \\
        X_{\vert \mathcal{M} \times \mathbb{R}} \in  \mathfrak{X}\left(\mathcal{M} \times \mathbb{R} \right).
    \end{dcases}
\end{equation}
where
$$ \ann{\mathfrak{D}^{l}} = \left( \pi_{Q,\mathbb{R}}\right)^{*} \flat_{g}\left( \ann{\mathcal{M}}\right)$$
Then, the integral curves of $X$ are solutions of the following constrained Hamiltonian Herglotz equations,
\begin{align}
    \begin{cases}\label{eq:Hamiltonnonholonomic_herglotz_eqs_coordsMechatype}
         \frac{\dd q^{i}}{\dd t} &=   \pdv{H}{p^{i}}\\
 \frac{\dd p_{i}}{\dd t} &=  - \pdv{H}{q^{i}}  -p_{i} \pdv{H}{z} - \lambda_a\Phi^a_i\\
 \frac{\dd z}{\dd t} &=   p_{i} \pdv{H}{p_{i}}-H\\
        \varphi^{a}\left(\flat_{g}\left(\dot{\xi}\right)\right) & = 0, \ \forall a
    \end{cases}
\end{align}

\end{theorem}
In this case, the constrained Hamiltonian vector field will be denoted by $X_{H, \mathcal{M}}$. Notice that

\begin{equation}\label{esta}
    \ann{\mathfrak{D}^{l}} = \left( \pi_{Q,\mathbb{R}}\right)^{*} \flat_{g}\left( \ann{\mathcal{M}}\right) = \left( \pi_{Q,\mathbb{R}}\right)^{*} \ann{\left(\sharp_{g}\left(\mathcal{M}\right)\right)} =  \left( \pi_{Q,\mathbb{R}}\right)^{*} \flat_{g}\left(\sharp_{g}\left(\mathcal{M}\right)^{\perp_{g}}\right)
\end{equation}

Therefore, Eq. (\ref{eq:Hamiltonnonholonomic_herglotz_eqs_coords}) (resp. Eq. (\ref{eq:Hamiltonnonholonomic_herglotz_eqs_coordsMechatype})) provides the correct nonholonomic dynamics in the Hamiltonian side. From now on, we will restrict to the case in which the constraint are codified into a vector subbundle $\mathcal{M}$ of $T^{*}Q$ (see Theorem \ref{18.3}). In other words, we will restrict to the constrained Hamiltonian systems given by the quadruples $\left( T^{*}Q \times \mathbb{R}, \eta_Q, H, \mathcal{M}\times \mathbb{R} , \mathfrak{D}^{l}\right)$, where $\mathcal{M}$ is a subbundle of the cotangent bundle $T^{*}Q$.\\

Let us now investigate the existence and uniqueness of solutions via Proposition \ref{GeneralizedYoootraproposiciondelcarajomas} and Corollary \ref{equaldimensions34}.

From the coordinate expression of the constraints $\tilde{\Phi}^a= \Phi^a_{i}\dd q^{i}$ defining $\mathfrak{D}^l$, one can see that $\Reeb_{Q}(\tilde{\Phi}^a)=0$, hence $\mathfrak{D}^l$ is vertical in the sense of~\cref{def:subspace_position}. Thus, it is satisfies condition $i)$ of Definition \ref{mechanicalcondit2343}.

Let $\set{Z_{a}}_{a=1 \dots k}$ be a local basis of $\ann{\mathcal{M}}$, and
$$ \flat_{g}\left( Z_{a}^{j} \frac{\partial}{\partial q^{j}}\right) \ = \ \left[Z_{a}^{j}g_{j,i}\right] \dd q^{i} ,$$
where $g_{i,j} = g \left(  \frac{\partial}{\partial q^{i}} , \frac{\partial}{\partial q^{j}}\right)$ are the local components of $g$. Then, $\tilde{\Phi}^{a} = \left[Z_{a}^{j}g_{j,i}\right]\dd q^{i}$ is a local basis of $\ann{\mathfrak{D}^{l}}$. Therefore, 
\begin{equation}\label{eq:324ff}
X_{a} = \sharp_{Q}\left( \tilde{\Phi}^{a}\right) = - \left[Z_{a}^{j}g_{j,i}\right] \frac{\partial}{\partial p_{i}}
\end{equation}
is a (local) basis of sections of the subbundle $\orthL{\mathfrak{D}^{l}}$. Observe that, it is trivial to prove that
\begin{equation}\label{33}
\orthL{\mathfrak{D}^{l}} \subseteq \mathfrak{D}^{l}
\end{equation} 

\begin{theorem}
The Hamiltonian Herglotz equations (\ref{6.3}) have the property of existence and uniqueness solutions.
\begin{proof}
Notice that $T \left( \mathcal{M} \times \RR \right)$ may be considered as a subbundle of $TQ \times \RR$ along the submanifold $ \mathcal{M} \times \RR $. Then, it is easy to show that the annihilator of the distribution $T \left( \mathcal{M} \times \RR \right)$ is given by $\pr_{\mathcal{M} \times \mathbb{R}}^{*} \ann{\left(T \mathcal{M} \right)}$ where $\pr_{\mathcal{M} \times \mathbb{R}}: \mathcal{M} \times \mathbb{R} \rightarrow \mathcal{M}$ denotes the projection on the first component. Then, let $Y =Y^{b}X_{b} $ be a vector field on $\mathcal{M} \times \RR$ tangent to $\orthL{\mathfrak{D}^{l}}$. Hence, we have that
\[ Y^{b} T Z_{a}\left( X_{b} \right) = 0.\]
Equivalently,
\[ Y^{b}Z_{b}^{j}g_{j,i} Z_{a}^{i} =  0, \ \forall a.\]
Then, $Y^{b}=0$ for all $b$.\\

On the other hand, for each $\left( \alpha_{q} , z \right) \in \mathcal{M} \times \RR$ we have that
\begin{itemize}
\item $\dim \left( \orthL{\mathfrak{D}^{l}}_{ \vert \left( \alpha_{q} , z \right)} \right) = k$

\item $\dim \left(T_{ \left( \alpha_{q} , z \right)} \left( \mathcal{M} \times \RR \right)\right) = 2n +1 -k $
\end{itemize}
So, uniqueness implies existence. In fact,
\begin{equation}\label{41}
\orthL{\mathfrak{D}^{l}} \oplus T \left( \mathcal{M} \times \RR \right) = T_{\mathcal{M} \times \RR}\left( T^{*}Q \times \RR \right),
\end{equation}
where $T_{\mathcal{M} \times \RR}\left( T^{*}Q \times \RR \right)$ consists of the tangent vectors to $T^{*}Q \times \RR $ at points of $\mathcal{M} \times \RR $.

\end{proof}
\end{theorem}

\noindent{Thus, we may consider these two projectors
\begin{subequations}\label{projectors2}
\begin{align}
\mathcal{P}: T_{\mathcal{M} \times \RR}\left( T^{*}Q \times \RR \right) &\rightarrow T \left(  \mathcal{M} \times \RR \right),\\
\mathcal{Q} :  T_{\mathcal{M} \times \RR}\left( T^{*}Q \times \RR \right) &\rightarrow \orthL{\mathfrak{D}^{l}}.
\end{align}
\end{subequations}
and the Hamiltonian constrained vector field $X_{H,\mathcal{M}}$ is obtained by projecting the Hamiltonian vector field $X_{H}$ with $\mathcal{P}$, i.e., $\mathcal{P} \left( {X_{H}}_{\vert \mathcal{M} \times\RR} \right) = X_{H , \mathcal{M}}$ is the solution of \cref{6.3}. Furthermore, an explicit (local) expression of the solution $X_{H, \mathcal{M}}$ is given by
\begin{equation}\label{20}
X_{H, \mathcal{M}} = \left(X_{H}\right)_{\vert \mathcal{M} \times \RR} -  X_{H}\left( Z_{b} \right)\mathcal{C}^{ba}X_{a}
\end{equation}

Recall that, by applying $\flat_{Q}$, we have a splitting on the cotangent bundle,
$$ T^{*}_{\mathcal{M} \times \RR}\left( T^{*}Q \times \RR \right) = \overline{T \left(\mathcal{M} \times \RR \right)} \oplus \ann{\mathfrak{D}^{l}},$$
in such a way that $\overline{T \left(\mathcal{M} \times \RR \right)} = \flat_{Q}\left( T \left(\mathcal{M} \times \RR \right)\right)$ and $T^{*}_{\mathcal{M} \times \RR} \left( T^{*}Q \times \RR \right)$ are the $1-$forms on $T^{*}Q \times \RR$ at points of $\mathcal{M} \times \RR$. Then, for all $a_{\left( \alpha_{q},z\right)} \in T^{*}_{\left( \alpha_{q},z\right)} \left( T^{*}Q \times \RR \right)$ with $\left( \alpha_{q},z\right) \in \mathcal{M} \times \RR$, the associated projections $\overline{\mathcal{Q}}: T^{*}_{\mathcal{M} \times \RR}\left( T^{*}Q \times \RR \right) \rightarrow \ann{\mathfrak{D}^{l}}$ and $\overline{\mathcal{P}}: T^{*}_{\mathcal{M} \times \RR}\left( T^{*}Q \times \RR \right) \rightarrow \overline{T \left(\mathcal{M} \times \RR \right)}$ are given by
\begin{itemize}
\item $\overline{\mathcal{Q}} \left( a_{\left( \alpha_{q},z\right)}\right) = \flat_{Q}\left(\mathcal{Q}\left(\sharp_{Q} \left( a_{\left( \alpha_{q},z\right)}\right)\right)\right)$.

\item $\overline{\mathcal{P}} \left( a_{\left( \alpha_{q},z\right)}\right) = \flat_{Q}\left(\mathcal{P}\left(\sharp_{Q} \left( a_{\left( \alpha_{q},z\right)}\right)\right)\right)$.
\end{itemize}
Furthermore, we have the following results, as corollaries of the correspondent ones in Section \ref{genframework24}.
\begin{theorem}\label{21}
Let $X$ be a vector field on $T^{*}Q \times \RR$. Then $X$ satisfies \cref{6.3}, if, and only if,
\begin{equation}
    \begin{dcases}
        \flat_{Q}\left( X \right) = \overline{\mathcal{P}} \left(  \dd H + \left(H + \Reeb_Q\left(H\right)\right)\eta_Q\right) \\
        X_{\vert \mathcal{M} \times \RR} \in  \mathfrak{X} \left(\mathcal{M} \times \RR \right).
    \end{dcases}
\end{equation}
\end{theorem}
\begin{theorem}\label{25}
A vector field $X$ on $T^{*}Q \times \mathbb{R}$ satisfies
$$\flat_{Q}\left(X\right) - \dd H + \left(H + \Reeb_{Q}\left(H\right)\right)\eta_Q \in \ann{\mathfrak{D}^{l}}$$
if and only if
\begin{equation}
    \begin{dcases}\label{26}
    \mathcal{L}_{X}\eta_{Q} + \Reeb_{Q}\left(H\right) \eta_{Q} \in \ann{\mathfrak{D}^{l}} \\
        \eta_{Q} \left(  X\right) = - H.
    \end{dcases},
\end{equation}

\end{theorem}

Notice that, by \cref{33}, we have that
\begin{eqnarray*}
\tilde{\Phi}_{a}\left( X_{H,\mathcal{M}}\right) &=& \tilde{\Phi}_{a}\left( {X_{H}}_{\vert \mathcal{M}\times \mathbb{R}}\right)\\
&=& \tilde{\Phi}_{a}\left(  \sharp_{Q}\left(\dd H - \left(H + \Reeb_Q\left(H\right)\right)\eta_Q \right)_{\vert \mathcal{M}\times \mathbb{R}}\right)\\
&=&\tilde{\Phi}_{a}\left(  \sharp_{Q}\left(\dd H \right)_{\vert \mathcal{M}\times \mathbb{R}}\right)
\end{eqnarray*}
Then, $X_{H,\mathcal{M}}- \sharp_{Q}\left(\dd H\right)_{\vert \mathcal{M}\times \mathbb{R}} \in \mathfrak{D}^{l}$. Hence, $X_{H,\mathcal{M}} \in \mathfrak{D}^{l}$ if, and only if, $\sharp_{Q}\left(\dd H\right)_{\vert \mathcal{M}\times \mathbb{R}} \in \mathfrak{D}^{l}$. On the other hand, for all $\left(\alpha_{q} , z\right) \in \mathcal{M} \times \mathbb{R}$, we have
\begin{eqnarray*}
\set{\tilde{\Phi}_{a}\left(  \sharp_{Q}\left(\dd H \right)_{\vert \mathcal{M}\times \mathbb{R}}\right)}\left(\alpha_{q} , z\right)
&=& \set{\Phi_{a}^{i} \frac{\partial H}{\partial p_{i}}}\left(\alpha_{q} , z\right)\\
&=& \set{\Phi^{i}_{a}g^{i,k}p_{k}}\left(\alpha_{q} , z\right)\\
&=& \Phi_{a}\left( \sharp_{g} \left( \alpha_{q}\right)\right)\\
&=&0
\end{eqnarray*}
Therefore,
\begin{equation}\label{24}
X_{H , \mathcal{M}} \in \mathfrak{D}^{l} \ \ \ \ \ \ , \ \ \ \ \ \dd H_{\vert \mathcal{M}\times \mathbb{R}} \in \flat \left( \mathfrak{D}^{l}\right) = \ann{\orth{\mathfrak{D}^{l}}}
\end{equation}
Namely, it is satisfied the mechanical condition (see Definition \ref{mechanicalcondit2343}). In other words, we have proved that $\left(H, T^{*}Q\times \RR, \mathcal{M}\times \RR, \mathfrak{D}^{l}\right)$ satisfies the \textit{mechanical condition} in Definition \ref{mechanicalcondit2343}, as well as, $\orthL{\mathfrak{D}^{l}} \subseteq \mathfrak{D}^{l}$, $\Reeb_{Q} \in \mathfrak{D}^{l}$, and conditions of existence and uniqueness of the solutions in Proposition \ref{GeneralizedYoootraproposiciondelcarajomas} and Corollary \ref{equaldimensions34}.

Notice that, the mechanical condition also implies that
$$ T \pi_{Q,\mathbb{R}}\left( X_{H , \mathcal{M}} \right) \in \mathcal{D}.$$

As a consequence of satisfying these conditions, we still may consider another projection on $T_{\mathcal{M}\times \mathbb{R}}\left(T^{*}Q \times \mathbb{R}\right)$ by the splitting,
$$T\left( \mathcal{M} \times \mathbb{R} \right) \cap \mathfrak{D}^{l} \oplus \left[T\left( \mathcal{M} \times \mathbb{R} \right) \cap \mathfrak{D}^{l}\right]^{\perp} = T_{\mathcal{M}\times \mathbb{R}}\left(T^{*}Q \times \mathbb{R}\right)$$
with the projections,
\begin{subequations}\label{otromasprojectors2cotg}
\begin{align}
\mathrm{P}: T_{\mathcal{M} \times \RR}\left( T^{*}Q \times \RR \right) &\rightarrow T\left( \mathcal{M} \times \mathbb{R} \right) \cap \mathfrak{D}^{l},\\
\mathrm{Q} :  T_{\mathcal{M} \times \RR}\left( T^{*}Q \times \RR \right) &\rightarrow \left[T\left( \mathcal{M} \times \mathbb{R} \right) \cap \mathfrak{D}^{l}\right]^{\perp}.
\end{align}
\end{subequations}
Therefore,
$$\mathrm{P} \left( {X_{H}}_{\vert \mathcal{M} \times\RR} \right) = \mathcal{P} \left( {X_{H}}_{\vert \mathcal{M} \times\RR} \right) = X_{H , \mathcal{M}}$$
Observe that, in this case, also happens that $\Reeb_{Q}\in T\left( \mathcal{M} \times \mathbb{R} \right)$. Then,
$$\orthL{\left[T\left( \mathcal{M} \times \mathbb{R} \right) \cap \mathfrak{D}^{l}\right]} = \left[T\left( \mathcal{M} \times \mathbb{R} \right) \cap \mathfrak{D}^{l}\right]^{\perp}.$$
Finally, we may define a \textit{non-holonomic bracket} for the constrained Hamiltonian system $\left(H, T^{*}Q\times \RR, \mathcal{M}\times \RR, \mathfrak{D}^{l}\right)$ given in \cref{gennonholonbracketdef56}.

\begin{remark}
\rm{
Recall that, in this case, the vector fields $Z_{a}$, as fiberwise linear map over $T^{*}Q \times \mathbb{R}$, define locally $\mathcal{M}\times \RR$, and \begin{equation*}
X_{a} = \sharp_{Q}\left( \tilde{\Phi}^{a}\right) = - \left[Z_{a}^{j}g_{j,i}\right] \frac{\partial}{\partial p_{i}}
\end{equation*}
is a (local) basis of sections of the subbundle $\orthL{\mathfrak{D}^{l}}$. Then, a basis of $\orthL{TM}$, is given by
\begin{eqnarray*}
Y_{c} &=&    Z_{c}^{i}  \frac{\partial}{\partial q_{i}}  -  p_{l}\frac{\partial  Z_{c}^{l}}{\partial q^{i}}\frac{\partial}{\partial p_{i}}   + p_{i}  Z_{c}^{i}  \frac{\partial}{\partial z}
\end{eqnarray*}
Therefore, the matrix $\mathcal{G}$ has the following coefficients
\begin{equation}\label{coordinatessecproy354}
\mathcal{G}_{dc} = \tilde{\Phi}^{d} \left(Y_{c}\right)  =  Z_{d}^{j}g_{j,i}Z^{i}_{c} = X_{d}\left( Z_{c}\right) = -\mathcal{C}_{dc}
\end{equation}

So, $\orthL{\left(TM \cap F\right)}$ is locally generated by the family of (local) vector fields $\set{Y_{c},X_{a}}$, where
\begin{equation*}
X_{a} = \sharp \left( {\Psi}^{a}\right) = \Psi^{a}_{i}\frac{\partial}{\partial q^{i}} - \left[\Phi^{a}_{i} + \mu^{a}p_{i}\right] \frac{\partial}{\partial p_{i}} + \left[\mu^{a}+\Psi^{a}_{k}p_{k}\right]\frac{\partial}{\partial z}
\end{equation*}
is a (local) basis of sections of the subbundle $\orthL{F}$. Consider a (local) basis $\{ \theta_{i}\}$ of $T M \cap F$. Then, the restriction of any (local) vector field $Y \in \mathfrak{X}\left(P\right)$ to $M$ may be locally written as
\[ Y_{\vert M} = R^{i}\theta_{i} + \nu^{c}Y_{c} + \lambda^{a}X_{a}.\]
Then, by applying $\Psi^{d}$ ($\orthL{F} \subseteq F$), we get
$$\Psi^{d} \left( Y \right) = \nu^{c} \Psi^{d} \left(Y_{c}\right)$$
Let us denote by $\mathcal{G}$ to the matrix with coefficients,
$$ \mathcal{G}_{dc}= \Psi^{d} \left(Y_{c}\right).$$
Hence,
$$\nu^{c} = \Psi^{d} \left( Y \right)\mathcal{G}^{dc},$$
with $\mathcal{G}^{dc}$, the coefficients of the inverse matrix of $\mathcal{G}$. On the other hand, applying $\dd \varphi^{b}$ we have that
\[ \dd \varphi^{b} \left( Y \right) = Y \left( \varphi^{b}\right) = \nu^{c}Y_{c}\left( \varphi^{b}\right) + \lambda^{a} \mathcal{C}_{ba},\]
where $\mathcal{C}$ is the matrix with coefficients given by \cref{matrixcoeff34}. Therefore,
$$\lambda^{a} = \left[Y \left( \varphi^{b}\right) -\nu^{c}Y_{c}\left( \varphi^{b}\right)\right] \mathcal{C}^{ba},$$
where $\mathcal{C}^{ba}$ are the coefficients of the inverse matrix of $\mathcal{C}$.
Hence,
\begin{align*}
\mathrm{Q} \left( Y_{\vert M} \right) &=   \Psi^{d} \left( Y \right)\mathcal{G}^{dc}Y_{c} + \left[Y \left( \varphi^{b}\right) -\Psi^{d} \left( Y \right)\mathcal{G}^{dc}Y_{c}\left( \varphi^{b}\right)\right] \mathcal{C}^{ba}X_{a}.\\
\mathrm{P} \left( Y_{\vert M} \right) &= Y_{\vert M} -  \Psi^{d} \left( Y \right)\mathcal{G}^{dc}Y_{c} - \left[Y \left( \varphi^{b}\right) -\Psi^{d} \left( Y \right)\mathcal{G}^{dc}sY_{c}\left( \varphi^{b}\right)\right] \mathcal{C}^{ba}X_{a}.
\end{align*}

}
\end{remark}

Notice that, in this case, the Reeb vector field $\Reeb_{Q}$ is also tangent to $\mathcal{M} \times\RR$. Then, the nonholonomic bracket is defined along $\mathcal{M} \times \RR$ as follows
\begin{equation}\label{Hamilnonholonbracketdef56}
  \set{f,g}_{H,\mathcal{M}} = \Lambda_{H,\mathcal{M}}(\dd f, \dd g) - f \Reeb_{Q}(g) + g \Reeb_{Q}(f).
\end{equation}
where
$$
    \Lambda_{H,\mathcal{M}} = \mathcal{P}_* {\Lambda_Q}_{\vert \mathcal{M} \times \RR},
$$
$\Lambda_{Q}$ being the Jacobi structure associated to the canonical contact form $\eta_{Q}$. Then, along $\mathcal{M} \times \RR$, we have
  \begin{equation*}
   X_{H, \mathcal{M}} =
    {\sharp_{\Lambda_{H,\mathcal{M}}}}(\dd H) - H \Reeb_{Q},
  \end{equation*}
Recall that Theorem \ref{genesteteorema23} shows us that this bracket may be restricted to $\mathcal{M}\times \RR$, and provides the correct evolution of the observables by the formula
\begin{equation}
      X_{H,\mathcal{M}}(f) = \set{H,f}_{H,\mathcal{M}} - f \Reeb_{Q} (H).
    \end{equation}

Let us describe locally this bracket. Let $\left(q^{i}\right)$ be a local system of coordinates on $Q$. Then,
\begin{itemize}
        \item $\mathcal{P}^{*}\left( \dd q^{i}\right) = \dd q^{i}$
        \item $\mathcal{P}^{*}\left( \dd p_{i}\right) = \dd p_{i} +g_{m,i}Z^{m}_{a}\mathcal{C}^{ba} p_{l} \frac{\partial Z_{b}^{l}}{\partial q^{j}} \ \dd q^{j}   +Z^{j}_{b}g_{l,i}Z^{l}_{a}\mathcal{C}^{ba} \dd p_{j}$
            \item $\mathcal{P}^{*}\left( \dd z\right) = \dd z$
\end{itemize}
This implies that
\begin{itemize}
    \item $\sharp_{Q} \left(\mathcal{P}^{*}\left( \dd q^{i}\right)  \right) = - \frac{\partial}{\partial p_{i}} $
    \item 
    $\sharp_{Q} \left( \mathcal{P}^{*}\left( \dd p_{i}\right)  \right)  = \frac{\partial}{\partial q_{i}} +\left[ Z^{j}_{b}g_{l,i}Z^{l}_{a}\mathcal{C}^{ba}  - g_{m,i}Z^{m}_{a}\mathcal{C}^{ba} p_{l} \frac{\partial Z_{b}^{l}}{\partial q^{j}}\right]\frac{\partial}{\partial q_{j}} + \left[p_{i} + Z^{j}_{b}g_{l,i}Z^{l}_{a}\mathcal{C}^{ba}
     p_{j} \right]\frac{\partial}{\partial z}$
    \item $\sharp_{Q} \left( \mathcal{P}^{*}\left( \dd z\right)  \right) =  p_{i} \frac{\partial}{\partial p_{i}} + \frac{\partial}{\partial z}$
\end{itemize}
So,
\begin{itemize}
    \item $\Lambda_{H , \mathcal{M}}\left( \dd q^{i}, \dd q^{j}\right)  = \Lambda_{H , \mathcal{M}}\left( \dd z, \dd z\right) = \Lambda_{H , \mathcal{M}}\left( \dd q^{i}, \dd z\right) = \Lambda_{H , \mathcal{M}}\left( \dd p_{i}, \dd p_{j}\right) = 0$.
    \item $\Lambda_{H , \mathcal{M}}\left( \dd p_{i}, \dd q^{j}\right) = \delta_{i}^{j} + Z^{j}_{b}g_{l,i}Z^{l}_{a}\mathcal{C}^{ba}  - g_{m,i}Z^{m}_{a}\mathcal{C}^{ba} p_{l} \frac{\partial Z_{b}^{l}}{\partial q^{j}}$
    \item $\Lambda_{H , \mathcal{M}}\left( \dd p_{i}, \dd z\right) = p_{i} + p_{r}\left(Z^{r}_{b}g_{l,i}Z^{l}_{a}\mathcal{C}^{ba}  - g_{m,i}Z^{m}_{a}\mathcal{C}^{ba} p_{l} \frac{\partial Z_{b}^{l}}{\partial q^{r}}\right)$
\end{itemize}
Therefore, by taking into account that $\Reeb_{Q} = \frac{\partial}{\partial z}$, we have that the local expression of the nonholonomic bracket is given by
\begin{itemize}
    \item $\set{q^{i}, q^{j}}_{H,\mathcal{M}}  = \set{p^{i}, p^{j}}_{H,\mathcal{M}} = \set{z, z}_{H,\mathcal{M}} =0$.
    \item $\set{p_{i}, q^{j}}_{H,\mathcal{M}} = \delta_{i}^{j} + Z^{j}_{b}g_{l,i}Z^{l}_{a}\mathcal{C}^{ba}  - g_{m,i}Z^{m}_{a}\mathcal{C}^{ba} p_{l} \frac{\partial Z_{b}^{l}}{\partial q^{j}}$
    \item $ \set{p_{i}, z}_{H,\mathcal{M}} =  p_{r}\left(Z^{r}_{b}g_{l,i}Z^{l}_{a}\mathcal{C}^{ba}  - g_{m,i}Z^{m}_{a}\mathcal{C}^{ba} p_{l} \frac{\partial Z_{b}^{l}}{\partial q^{r}}\right)$
    \item $\set{q^{i}, z}_{H,\mathcal{M}} = -q^{i}$
\end{itemize}

\subsection{The contact Eden bracket}

We can also define another nonholonomic bracket, based in the projection derived from the Riemaniann metric.

Notice that,
$$TQ = \mathcal{D}\oplus \mathcal{D}^{\perp_{g}},$$
where $\mathcal{D}= \sharp_{Q}\left( \mathcal{M}\right)$. Then, by applying $\flat_{g}$, we have that
$$T^{*}Q = \mathcal{M}\oplus \ann{\mathcal{D}}$$
In particular,
$$T^{*}Q\times \mathbb{R} = \left(\mathcal{M}\times \mathbb{R}\right)\oplus \ann{\mathcal{D}}$$
Let us consider the projection $\gamma: T^{*}Q\times \mathbb{R} \rightarrow \mathcal{M}\times \mathbb{R}$. Then, it turns natural to define the following bracket of functions,
$$\set{f,g}_{E} = \set{f \circ \gamma,g\circ \gamma}_{\vert \mathcal{M}\times \mathbb{R}}$$
for all $f,g \in \mathcal{C}^{\infty}\left(\mathcal{M}\times \mathbb{R}\right)$. This bracket will be called \textbf{contact Eden bracket} following the spirit of \cite{eden1,eden2}.\\
Let $Z \in T_{\left( \alpha_{q} ,s \right)} \left(T^{*}Q \times \mathbb{R}\right)$, $\alpha_{q} \in \mathcal{M}$. Consider a curve $\left(\alpha \left( t\right), f \left( t \right) \right) \subseteq T^{*}Q\times \mathbb{R}$ in such a way that $\left(\dot{\alpha}\left( 0\right) , \dot{f}\left( 0 \right) \right)= Z$, $\alpha \left( 0 \right) = \alpha_{q}$ and $f\left( 0 \right) = s$. So, we may define
$$ \Lambda \left( t \right) = \left(\alpha \left( t\right), f \left( t \right) \right) - \gamma \left( \left(\alpha \left( t\right), f \left( t \right) \right) \right) \in \ann{\mathcal{D}}_{b\left( t \right)},$$
where $b \left( t \right) = \pi_{Q,\mathbb{R}}\left(\alpha \left( t\right), f \left( t \right) \right) = \pi_{Q,\mathbb{R}}\left( \gamma\left(\alpha \left( t\right), f \left( t \right) \right) \right)$. Notice that $\gamma$ leaves the coordinate on $\mathbb{R}$ invariant. 

Then,
\begin{equation}\label{constraintsnoh34}
\Lambda \left( t \right) = \lambda_{a}\left( t \right)\Phi^{a}\left( b\left( t \right) \right) =\lambda_{a}\left( t \right)\Phi^{a}_{i}\left( b\left( t \right) \right) dq^{i}_{\vert b \left( t \right)} .
\end{equation}
Observe that $\Lambda \left( 0 \right) = 0_{q}\in T^{*}_{q}Q$, which implies that $\lambda_{a}\left( 0\right) = 0$, for all $a$ (since $\set{\Phi^{a}}$ is a basis of $\ann{\mathcal{D}}$). Therefore, by \cref{constraintsnoh34}, we have in coordinates
\begin{eqnarray*}
\dot{\Lambda}\left( 0 \right) &=& \dot{b}\left( 0 \right)  \frac{\partial}{\partial {q_{i}}}_{\vert 0_{q}} +    \left[\dot{\alpha}^{i}\left(0\right) - \dot{\gamma\left(\alpha\right)}^{i}\left(0\right) \right] \frac{\partial}{\partial {p_{i}}}_{\vert 0_{q}}\\
&=& \dot{b}\left( 0 \right)  \frac{\partial}{\partial {q_{i}}}_{\vert 0_{q}} +    \dot{\lambda}_{a}\left( 0 \right)\Phi^{a}_{i}\left(  \alpha_{q} \right) \frac{\partial}{\partial {p_{i}}}_{\vert 0_{q}}
\end{eqnarray*}
Notice that we have used that $\lambda_{a}\left( 0\right) = 0$, for all $a$.\\
On the other hand, observe that
\begin{align*}
 Z  &= \dot{b}\left( 0 \right)  \frac{\partial}{\partial {q_{i}}}_{\vert \left( \alpha_{q} ,s \right)} +    \dot{\alpha}^{i}\left(0\right) \frac{\partial}{\partial {p_{i}}}_{\vert \left( \alpha_{q} ,s \right)}\\
T\gamma \left( Z \right) &= \dot{b}\left( 0 \right)  \frac{\partial}{\partial {q_{i}}}_{\vert \left( \alpha_{q} ,s \right)} +    \dot{\gamma\left(\alpha\right)}^{i}\left(0\right) \frac{\partial}{\partial {p_{i}}}_{\vert \left( \alpha_{q} ,s \right)}
\end{align*}
where $\alpha \left( t \right) = \alpha^{i}\left(t\right) \frac{\partial}{\partial {q_{i}}}_{\vert b\left( t \right)}$, and $\gamma\left(\alpha\right)\left( t \right) = \gamma\left(\alpha\right)^{i}\left(t\right) \frac{\partial}{\partial {q_{i}}}_{\vert b\left( t \right)}$. Thus,
\begin{eqnarray*}
    X= Z-T\gamma \left( Z \right) &=&  \left[\dot{\alpha}^{i}\left(t\right) - \dot{\gamma\left(\alpha\right)}^{i}\left(t\right) \right]\frac{\partial}{\partial {p_{i}}}_{\vert \left( \alpha_{q} ,s \right)}\\
    &=& \dot{\lambda}_{a}\left( 0 \right)\Phi^{a}_{i}\left(  \alpha_{q} \right) \frac{\partial}{\partial {p_{i}}}_{\vert \left( \alpha_{q} ,s \right)}
\end{eqnarray*}
Observe that $\eta_{Q}\left( X \right) = 0$. Therefore,
$$ \flat_{Q} \left( X \right) = \contr{X}\dd \eta_{Q} = - X\left( p_{i}\right) dq^{i}_{\vert \left( \alpha_{q} ,s \right)} = - \dot{\lambda}_{a}\left( 0 \right)\Phi^{a}_{i}\left(  \alpha_{q} \right)  dq^{i}_{\vert \left( \alpha_{q} ,s \right)}$$
In other words, $\flat_{Q} \left( X \right) \in  \ann{\mathfrak{D}^{l}}$ or, equivalently, $X \in \orthL{\mathfrak{D}^{l}}$. Thus, we will prove the following result
\begin{theorem}
Let be $\alpha_{q} \in \mathcal{M}$. Then,
$$ T \gamma\left( Z \right) \ = \ \mathcal{P}\left( Z \right),$$
for all $Z \in T_{ \left( \alpha_{q} ,s \right)} \left(T^{*}Q \times \mathbb{R}\right)$.
\begin{proof}
Take into account that $Z = T \gamma \left(Z \right) + X$, with $X= Z-T\gamma \left( Z \right) \in \orthL{\mathfrak{D}^{l}}$. Furhtermore, obviously $T \gamma \left(Z \right) \in T \left( \mathcal{M}\times \mathbb{R}\right)$. Then, by uniqueness, we obtain the result.  
\end{proof}
\end{theorem}
Notice that, as an immediate corollary, we deduce
\begin{corollary}
$T\gamma \left( {X_{H}}_{\vert \mathcal{M} \times\RR} \right) = X_{H , \mathcal{M}}$
\end{corollary}
Recall that the adjoint operator of $\mathcal{P}$ is given by, $\mathcal{P}^* :T^{*}_{\mathcal{M} \times \RR} \left( T^{*}Q \times \RR \right) \rightarrow \ann{\orthL{\mathfrak{D}^{l}}}$, such that, for all $\alpha \in T_{\left(\alpha_{q} , s\right)}^{*} \left( T^{*}Q \times \RR \right) $ and $Z \in T_{\left(\alpha_{q} , s\right)} \left( T^{*}Q \times \RR \right)$, we have
\begin{eqnarray*}
  \mathcal{P}^{*}\alpha \left( Z \right) &=& \alpha \left( \mathcal{P}\left( Z\right)\right)\\
  &=& \alpha \left( T \gamma\left( Z\right)\right)\\
  &=& \set{T^{*}\gamma \left(\alpha\right)}  \left( Z\right)
\end{eqnarray*}
Hence, $T^{*}\gamma = \mathcal{P}^{*}$.
\begin{theorem}
The contact Eden bracket $\set{\cdot, \cdot}_{E}$ and the nonholonomic bracket $\set{\cdot, \cdot}_{H,\mathcal{M}}$ coincide, namely
$$\set{f,g}_{E} = \set{f \circ \gamma,g\circ \gamma}_{\vert \mathcal{M}\times \mathbb{R}}$$

\begin{proof}
Recall that the nonholonomic bracket was defined on functions over $\mathcal{M} \times \RR$ as follows
\begin{equation*}
  \set{f,g}_{H,\mathcal{M}} = \Lambda_{H,\mathcal{M}}(\dd f, \dd g) - f \Reeb_{H,\mathcal{M}}(g) + g \Reeb_{H,\mathcal{M}}(f).
\end{equation*}
where
\begin{align*}
    \Reeb_{H,\mathcal{M}} &=  \mathcal{P} \left({\Reeb_{Q}}_{\vert \mathcal{M} \times \RR}\right),\\
    \Lambda_{H,\mathcal{M}} &= \mathcal{P}_* {\Lambda_Q}_{\vert \mathcal{M} \times \RR},
\end{align*}
with $\Lambda_{Q}$ as the Jacobi structure associated to the contact form $\eta_{Q}$. So, for $f,g \in \mathcal{M}\times \mathbb{R}$, we have that
\begin{eqnarray*}
\Lambda_{H,\mathcal{M}}\left(\dd f, \dd g\right) &=& {\Lambda_{Q}}_{\vert \mathcal{M}\times \mathbb{R}}\left(\mathcal{P}^*\left(\dd f\right),\mathcal{P}^*\left(\dd g\right)\right)\\
&=& {\Lambda_{Q}}_{\vert \mathcal{M}\times \mathbb{R}}\left(T^{*}\gamma \left(\dd f\right),T^{*}\gamma \left(\dd g\right)\right)\\
&=& {\Lambda_{Q}}_{\vert \mathcal{M}\times \mathbb{R}}\left(\dd \left(f\circ \gamma\right),\dd \left(g\circ \gamma\right)\right)
\end{eqnarray*}
On the other hand,
$$\Reeb_{H,\mathcal{M}} \left( f \right)  =   \mathcal{P} \left({\Reeb_{Q}}_{\vert \mathcal{M} \times \RR} \right) \left( f \right) = \set{T \gamma  \left({\Reeb_{Q}}_{\vert \mathcal{M} \times \RR}\right)} \left( f \right) = {\Reeb_{Q}}_{\vert \mathcal{M} \times \RR}\left( f \circ \gamma\right)$$
Therefore,
\begin{eqnarray*}
 \set{f,g}_{H,\mathcal{M}}  &=& \Lambda_{H,\mathcal{M}}(\dd f, \dd g) - f \Reeb_{H,\mathcal{M}}(g) + g \Reeb_{H,\mathcal{M}}(f)\\
 &=& {\Lambda_{Q}}_{\vert \mathcal{M}\times \mathbb{R}}\left(\dd \left(f\circ \gamma\right),\dd \left(g\circ \gamma\right)\right) \\
 &&- f {\Reeb_{Q}}_{\vert \mathcal{M} \times \RR}\left( g \circ \gamma\right)+ g {\Reeb_{Q}}_{\vert \mathcal{M} \times \RR}\left( f \circ \gamma\right)\\
 &=& \set{f \circ \gamma,g\circ \gamma}_{\vert \mathcal{M}\times \mathbb{R}}
\end{eqnarray*}
Notice that, the last equality is a consequence of this fact $\left(f \circ \gamma \right)_{\vert \mathcal{M}\times \mathbb{R}} = f$. So,

$$\set{f,g}_{E} = \set{f \circ \gamma,g\circ \gamma}_{\vert \mathcal{M}\times \mathbb{R}}$$
\end{proof}
\end{theorem}

At this point, it is important to note that we have restricted our results to a Hamiltonian constrained systems induced by a Lagrangian of mechanical type. However, we actually do not need the existence of a Riemannian metric to construct the Eden bracket, and to prove the associated results. In fact, it is enough to consider a constrained Hamiltonian systems given by the quintuple $\left(  T^{*}Q \times \mathbb{R}, \eta_{Q}, H, \mathcal{M}\times \mathbb{R} , \mathfrak{D}^{l}\right)$, where $\mathcal{M}$ is a subbundle of the cotangent bundle $T^{*}Q$ in such a way that
$$T^{*}Q = \mathcal{M}\oplus \ann{\mathcal{D}}$$
So, in general, we may easily generalize the Eden bracket, for the case in which our constrained Hamiltonian systems is generated by a splitting of the cotangent bundle,
$$T^{*}Q = \mathcal{M}\oplus \mathcal{N}$$
where $\mathcal{M}$ and $\mathcal{N}$ are subbundles of $T^{*}Q$. Then, the constrained Hamiltonian system is defined by the quadruples $\left( T^{*}Q \times \mathbb{R}, \eta_Q, H, \mathcal{M}\times \mathbb{R} , \mathfrak{N}^{l}\right)$, where
$$\left( \pi_{Q,\mathbb{R}}\right)^{*} \mathcal{N} = \ann{\mathfrak{N}^{l}}$$

\section{Conclusions and further work}

In this paper, we have presented a general framework of contact non-holonomic dynamics in which we have studied conditions of uniqueness and existence of solutions. Furthermore, we have been able to define two equivalent non-holonomic brackets describing the dynamics of the system. We have proved that the known case of constraint dynamics for Lagrangian systems studied in \cite{primero} can be obtained as a particular case of this framework. We have also present the Hamiltonian counterpart, which is another a particular case of this framework. Finally, in this context, we have been able to describe another version of the non-holonomic bracket, which is the so-called contact Eden bracket, which turns out to be equivalent to the two already described before.

These are some of the subjects on which we wish to work in the near future and which would be a natural continuation of the results contained in this paper: 

\begin{enumerate}

\item In a future paper we will extend the reduction procedure developed in \cite{cantrijn1999reduction,CoLe99} for contact nonholonomic mechanical systems.

\item We want to develop numerical integrators for this type of systems, extending the results of \cite{Jorge,aitor}.

\item We will develop a noholonomic moment map in this case and the relationship with the noholonomic bracket \cite{cantrijn1999reduction}.
    
\end{enumerate}

\section*{Acknowledgments}

M. de Le\'on and V. M. Jiménez acknowledge financial support from the Spanish Ministry of Science and Innovation, under grants Grant PID2019-106715GB-C21 and PID2022-137909NB-C2 and the Severo Ochoa Programme for Centres of Excellence in R$\&$D” (CEX2019-000904-S). V. M. Jiménez acknowledge the finantial support from MIU and European Union-\textit{NextGenerationUE}.


\bibliographystyle{plain}

\bibliography{Library}

\end{document}